\documentclass[10pt,a4paper]{amsart}

\usepackage{tikz-cd}
\usepackage[T2A]{fontenc}
\usepackage[utf8]{inputenc}
\usetikzlibrary{calc}
\usepackage{comment}
\usepackage[all]{xy}

\usepackage[bottom]{footmisc}

\usepackage{amsfonts}
\usepackage{amsthm}
\usepackage{amsmath}
\usepackage{amscd}

\usepackage{t1enc}

\usepackage{indentfirst}
\usepackage{graphicx}
\usepackage{graphics}
\usepackage{pict2e}
\usepackage{epic}
\numberwithin{equation}{section}
\usepackage[margin=2.9cm]{geometry}
\usepackage{epstopdf} 
\usepackage{amscd}
\usepackage{sseq}
\usepackage{amsmath}
\usepackage{amssymb}
\usepackage{ dsfont }
\usepackage{textcomp}
\usepackage{amsthm}

\usepackage{graphicx}
\graphicspath{ {./images/} }

\usepackage{euscript}
\usepackage{ mathrsfs }

\newcommand{\set}[1]{\mathbb{#1}}

\usepackage{amsfonts}
\usepackage{amsrefs}
\usepackage{amsthm}
\usepackage{mathrsfs}
\usepackage[pdfencoding=auto, psdextra]{hyperref}
\usepackage{amsmath}

\theoremstyle{plain}
\newtheorem{Th}{Theorem}[section]

\newtheorem{Lemma}[Th]{Lemma}
\newtheorem{Cor}[Th]{Corollary}
\newtheorem{Prop}[Th]{Proposition}

\newtheorem{Warning}[Th]{Warning}

\theoremstyle{definition}
\newtheorem{Def}[Th]{Definition}
\newtheorem{Conj}[Th]{Conjecture}
\newtheorem{Ass}[Th]{Assumption}

\newtheorem{Example}[Th]{Example}
\newtheorem{Rem}[Th]{Remark}

\newtheorem{?}[Th]{Problem}

\usepackage{etoolbox}
\usepackage{float}

\AtBeginEnvironment{Rem}{\footnotesize}

\renewcommand{\C}{\mathbb{C}}

\newcommand{\M}[0]{\mathbf{M}}
\newcommand{\vv}[0]{\mathbf{v}}
\newcommand{\ww}[0]{\mathbf{w}}
\newcommand{\bla}[0]{\boldsymbol{\la}}

\newcommand{\nc}[0]{\newcommand}
\nc{\on}[0]{\operatorname}
\nc{\BC}[0]{\mathbb{C}}
\nc{\BZ}[0]{\mathbb{Z}}
\nc{\al}[0]{\alpha}
\nc{\ul}[0]{\underline}
\nc{\la}[0]{\lambda}
\nc{\CK}[0]{\mathcal{K}}
\nc\iso{\,\vphantom{j^{X^2}}\smash{\overset{\sim}{\vphantom{\rule{0pt}{0.20em}}\smash{\longrightarrow}}}\,}
\nc{\ra}{\rightarrow}
\nc{\ol}{\overline}
\nc{\La}{\Lambda}
\nc{\CH}{\mathcal{H}}
\nc{\CP}{\mathcal{P}}
\nc{\BA}{{\mathbb{A}}}
\nc{\BT}{\mathbb{T}}
\nc{\CM}{\mathcal{M}}
\nc{\BY}{\mathbb{Y}}
\nc{\BH}{\mathbb{H}}
\nc{\CN}{\mathcal{N}}
\nc{\CC}{\mathcal{C}}
\nc{\CO}{\mathcal{O}}
\nc{\CR}{\mathcal{R}}
\nc{\BP}{\mathbb{P}}
\nc{\CV}{\mathcal{V}}
\nc{\CE}{\mathcal{E}}

\begin{document}
\title{Hikita-Nakajima conjecture for the Gieseker variety}
\author{Vasily Krylov }
\address{Department of Mathematics
Massachusetts Institute of Technology
\newline
77 Massachusetts Avenue,
Cambridge, MA 02139,
USA;
\newline National Research University Higher School of Economics, Russian Federation\newline
Department of Mathematics, 6 Usacheva st., Moscow 119048;
}
\email{krvas@mit.edu, krylovasya@gmail.com}

\author{Pavel Shlykov}
\address{Department of Mathematics University of Toronto
\newline 40 St George St, Toronto, ON M5S 2E4, Canada
\newline National Research University Higher School of Economics, Russian Federation\newline
Department of Mathematics, 6 Usacheva st., Moscow 119048;
}
\email {pavel.shlykov@mail.utoronto.ca, pavelshlykov57@gmail.com}

\begin{abstract}
Let $\mathfrak{M}_0$ be an affine Nakajima quiver variety, and let $\mathcal{M}$ be the corresponding BFN Coulomb branch. Assume that $\mathfrak{M}_0$ can be resolved by the (smooth) Nakajima quiver variety $\mathfrak{M}$.  The Hikita-Nakajima conjecture claims that there should be an isomorphism of (graded) algebras $H^*_{S}(\mathfrak{M},\mathbb{C}) \simeq \mathbb{C}[\mathcal{M}_{\mathfrak{s}}^{\mathbb{C}^\times}]$, where $S \curvearrowright~\mathfrak{M}_0$ is a torus acting on $\mathfrak{M}_0$ preserving the Poisson structure, $\mathcal{M}_{\mathfrak{s}}$ is the (Poisson) deformation of $\mathcal{M}$ over $\mathfrak{s}=\on{Lie}S$, $\mathbb{C}^\times$ is a generic one-dimensional torus acting on $\mathcal{M}$, and $\mathbb{C}[\mathcal{M}_{\mathfrak{s}}^{\mathbb{C}^\times}]$ is the algebra of schematic $\mathbb{C}^\times$-fixed points of $\mathcal{M}_{\mathfrak{s}}$. We prove the Hikita-Nakajima conjecture for $\mathfrak{M}=\mathfrak{M}(n,r)$ Gieseker variety ($ADHM$ space). We produce the isomorphism explicitly on generators. 

We also describe the Hikita-Nakajima isomorphism above using the realization of $\mathcal{M}_{\mathfrak{s}}$ as
the spectrum of the center of the rational Cherednik algebra corresponding to $S_n \ltimes (\mathbb{Z}/r\mathbb{Z})^n$ and identify all the algebras that appear in the isomorphism with the center of the degenerate cyclotomic Hecke algebra (generalizing some results of Shan, Varagnolo, and Vasserot). 

\end{abstract}

\maketitle

\section{Introduction}

\subsection{Conical symplectic resolutions and symplectic duality}
Let $\mathfrak{M}_{0}$ be an algebraic variety over $\BC$. The following definition belongs to~\cite{BEA}.

\begin{Def}\label{sympl_sing_defe}
We say that $\mathfrak{M}_{0}$ is \emph{singular symplectic} (or has symplectic singularities) if: 

(1) $\mathfrak{M}_{0}$ is a normal Poisson variety, 

(2) there exists a smooth, dense open subset $U \subset \mathfrak{M}_{0}$ on which the Poisson structure comes from the symplectic form that we denote by $\omega$,

(3) there exists a resolution of singularities (birational and projective morphism) $\mathfrak{M}\ra \mathfrak{M}_{0}$ such that the pullback of $\omega$ to $\mathfrak{M}$ has no poles. 
\end{Def}

We say that $\mathfrak{M}_{0}$ is a \emph{conical symplectic singularity} if in addition to $(1) - (3)$ one has a $\BC^\times$-action on $\mathfrak{M}_{0}$ which acts on $\omega$ with some positive weight and contracts $\mathfrak{M}_{0}$ to the unique fixed point. 
We will denote the contracting $\BC^\times$ by $\BC^\times_\hbar$.

Assume now that $\mathfrak{M}_{0}$ possesses a $\BC^\times_\hbar$-equivariant symplectic resolution $\mathfrak{M} \ra \mathfrak{M}_{0}$. We will call $\mathfrak{M}\ra \mathfrak{M}_{0}$ a {\emph{conical symplectic resolution}}. 
It is known (see~\cite[Lemmas 12, 22, Proposition 13]{NAM4}, \cite[Theorem 1.13]{GK}) that there exist canonical symplectic (resp. Poisson) deformations of $\mathfrak{M}$ (resp. $\mathfrak{M}_{0}$) over the base $\mathfrak{t}=\mathfrak{t}_{\mathfrak{M}_{0}}:=H^2(\mathfrak{M},\BC)$ that we 
denote by  
\begin{equation*}
\mathfrak{M}_{\mathfrak{t}} \ra \mathfrak{t},\, \mathfrak{M}_{0,\mathfrak{t}} \ra \mathfrak{t}.
\end{equation*}

\begin{Rem}
{\emph{One can show that the space $\mathfrak{t}$ does not depend on the choice of the resolution $\mathfrak{M}$ of $\mathfrak{M}_{0}$. This space can be defined even if $\mathfrak{M}_{0}$ does not have a symplectic resolution $\mathfrak{M}$. The deformation $\mathfrak{M}_{\mathfrak{t}}$ is the universal deformation of $\mathfrak{M}$, $\mathfrak{M}_{0,\mathfrak{t}}$ is a pullback of the universal deformation of $\mathfrak{M}_{0}$ along the morphism $\mathfrak{t} \ra \mathfrak{t}/W$ for a certain group $W$ acting on $\mathfrak{t}$ (called the Namikawa-Weyl group).}}
\end{Rem}

Let $\on{Aut}_{\BC^\times_\hbar}(\mathfrak{M}_{0})$ be the group of Poisson automorphisms of $\mathfrak{M}_{0}$ commuting with the contracting $\BC^\times_\hbar$. This is a finite-dimensional algebraic group (possibly disconnected). We denote by $S=S_{\mathfrak{M}_{0}} \subset \on{Aut}_{\BC^\times_\hbar}(\mathfrak{M}_{0})$ a maximal torus and set $\mathfrak{s}_{\mathfrak{M}_{0}}:=\on{Lie}S_{\mathfrak{M}_{0}}$. One can show that the action of $S$ on $\mathfrak{M}_{0}$ and the contracting action of $\BC^\times_\hbar$ extend naturally to the  action of $S \times \BC^\times_\hbar$ on $\mathfrak{M}_{0,\mathfrak{t}}$ (the torus $S$ acts fiberwise). It can also be shown that the $S \times \BC^\times_\hbar$-action above lifts to the action on $\mathfrak{M}_{\mathfrak{t}}$.

Often conical singularities  come in ``dual'' pairs 
\begin{equation*}
\mathfrak{M}_{0},\, ~\mathfrak{M}_{0}^!.
\end{equation*}
We refer the reader to~\cite{BLPW} for  details on symplectic duality.

\begin{Rem}

It seems that there is no functorial way to define $\mathfrak{M}^!_{0}$ from $\mathfrak{M}_0$. Apparently one needs some richer mathematical structure to produce a
pair of symplectically dual varieties. For example it is expected that
every $\on{3d} \ \mathcal{N}=4$ super-symmetric quantum field theory should produce such
a pair.
In the case of gauge theories such a pair was produced mathematically in \cite{BFN}.
\end{Rem}

Let us recall that for symplectically dual varieties 
$\mathfrak{M}_{0}$, $\mathfrak{M}^!_{0}$
one should have the natural identifications 
\begin{equation*}
\mathfrak{t}_{\mathfrak{M}_{0}} \simeq \mathfrak{s}_{\mathfrak{M}_{0}^!},\, \mathfrak{s}_{\mathfrak{M}_{0}} \simeq \mathfrak{t}_{\mathfrak{M}_{0}^!}. \end{equation*}
Also it is expected that the choice of $\mathfrak{M}$ (choice of the symplectic resolution of $\mathfrak{M}_{0}$) should correspond on the dual side to a cocharacter $\nu_{\mathfrak{M}}\colon \BC^\times \ra S_{\mathfrak{M}_{0}^!}$ (actually to a choice of a certain chamber in $\on{Hom}(\BC^\times,T) \otimes_{\BZ} \mathbb{R}$) such that the set $(\mathfrak{M}_{0}^!)^{\nu_{\mathfrak{M}}(\BC^\times)}$ consists of one point.

In this paper, we restrict ourselves to the case when $\mathfrak{M}$ is a (smooth) Nakajima quiver variety, $\mathfrak{M}_0$ is the (affine)  Nakajima quiver variety, and $\mathfrak{M}_0^!=\mathcal{M}$ is the  BFN Coulomb branch of the corresponding quiver gauge theory (see \cite{BFN} for the definition of $\mathcal{M}$).

\subsection{Schematic fixed points and Hikita-Nakajima conjectures}

\subsubsection{Schematic fixed points} Given a variety $Y$ with an action of some algebraic group $G$ we can define the functor \begin{equation*}
Y^G\colon {\bf{Schemes}}_{\BC} \ra {\bf{Sets}}
\end{equation*} from the category ${\bf{Schemes}}_{\BC}$ of schemes over $\BC$ to the category ${\bf{Sets}}$ of sets as follows (see~\cite{FO},
\cite[Section 1.2]{Drinfeld}):
\begin{equation*}
Y^{G}(S):={\on{Maps}}^{G}(S,Y), \, S \in {\bf{Schemes}}_{\BC}, 
\end{equation*}
where the action of $G$ on $S$ is trivial and $\on{Maps}^G(S,Y)$ is the set of $G$-equivariant morphisms from $S$ to $Y$.
It turns out that in some cases, functor $Y^G$ is represented by a scheme that we call {\emph{schematic fixed points of $Y$}} (for more details, see~\cite[Theorem 2.3]{FO}). 
Consider the case $Y=\on{Spec}B$ for some $\BC$-algebra $B$ and $G=\BC^\times$. Then the action $\BC^\times \curvearrowright Y$ corresponds to the $\BZ$-grading $B=\bigoplus_{i \in \BZ}B_i$. 
Compare the following proposition with \cite[Example 1.2.3]{Drinfeld}.

\begin{Prop}\label{shematic_fixed_for_affine}
If $Y=\on{Spec}B$ is an affine variety and $G=\BC^\times$, then $Y^{{\BC^\times}}$ is represented by an affine scheme whose ring of functions can be described in two equivalent ways:
\begin{equation}\label{funct_fixed_point!}
\BC[Y^{\BC^\times}]=B_0/\sum_{i>0}B_{-i}B_{i}=B/(b_i \in B_i,\, i \neq 0). 
\end{equation}
\end{Prop}
\begin{proof}
It is enough to show that the functor $Y^{\BC^\times}$ restricted to the category of affine schemes over $\BC$ is represented by the affine scheme with the algebra of functions as in~(\ref{funct_fixed_point!}).
Let $S=\on{Spec}C$ be an affine scheme with trivial $\BC^\times$-action. The set $\on{Maps}^{\BC^\times}(S,Y)$ can be identified with the set of graded homomorphisms $B \ra C$, where the grading on $C$ is the trivial one ($C=C_0$). Since $C=C_0$, we conclude that every such  homomorphism $f\colon B \ra C$ factors through $B/(b_i \in B_i,\, i \neq 0)$.  Note now that every homomorphism $\ol{f}\colon B/(b_i \in B_i,\, i \neq 0) \ra C$ induces the {\emph{graded}} homomorphism $B \twoheadrightarrow B/(b_i \in B_i,\, i \neq 0) \ra C$, so we must have $Y^{\BC^\times}=\on{Spec}(B/(b_i \in B_i,\, i \neq 0))$.

It remains to note that the natural morphism 
\begin{equation*}
B_0/\sum_{i>0}B_{-i}B_{i} \ra B/(b_i \in B_i,\, i \neq 0),
\end{equation*}
given by 
\begin{equation*}
B_0/\sum_{i>0}B_{-i}B_{i} \ni [b] \mapsto [b] \in B/(b_i \in B_i,\, i \neq 0)   
\end{equation*}
is an isomorphism.
\end{proof}

\begin{Rem}
{\emph{Proposition~\ref{shematic_fixed_for_affine} can be easily generalized to the case when $G$ is a torus of arbitrary rank (or an arbitrary reductive group).
}}
\end{Rem}

The following proposition holds (see, for example,~\cite{Iversen}).
\begin{Prop}\label{Y_smooth_implies_fixed_smooth}
If $Y$ is a smooth algebraic variety over $\BC$ and $G$ is reductive, then $Y^G$ is smooth.
\end{Prop}

\subsubsection{Hikita-Nakajima conjecture}
Recall that $\mathfrak{M}_{0}$ is a Nakajima quiver variety and $\mathcal{M}=\mathfrak{M}_{0}^!$ is the corresponding Coulomb branch. We assume that $\mathfrak{M}_{0}$ is resolved by the corresponding (smooth) Nakajima quiver variety $\mathfrak{M}\ra \mathfrak{M}_{0}$. Let  $H^*_{S_{\mathfrak{M}_{0}}}(\mathfrak{M},\BC)$ be the algebra of $S_{\mathfrak{M}_{0}}$-equivariant cohomology of $\mathfrak{M}$.  This is an algebra over $H^*_{S_{\mathfrak{M}_{0}}}(\on{pt})=\BC[\mathfrak{s}_{\mathfrak{M}_{0}}]$.
Recall also that the choice of $\mathfrak{M}$ (resolution of $\mathfrak{M}_{0}$) corresponds to a (generic) cocharacter $\nu_\mathfrak{M}\colon \BC^\times \ra S_{\mathcal{M}}$.
We can consider the algebra of functions of schematic fixed points $\BC[(\mathcal{M}_{\mathfrak{t}_{\mathcal{M}}})^{\nu_{\mathfrak{M}}(\BC^\times)}]$ that is an algebra over $\BC[\mathfrak{t}_{\mathcal{M}}]=\BC[\mathfrak{s}_{\mathfrak{M}_{0}}]$. 
 
Note also that the algebra $H^*_{S_{\mathfrak{M}_{0}}}(\mathfrak{M},\BC)$ has a natural cohomological grading and the algebra   
$\BC[(\mathcal{M}_{\mathfrak{t}_{\mathcal{M}}})^{\nu_{\mathfrak{M}}(\BC^\times)}]$ is graded via the contracting $\BC^\times_{\hbar}$-action. The algebra $H^*_{S_{\mathfrak{M}_{0}}}(\mathfrak{M},\BC)$ is finitely generated over $\BC[\mathfrak{s}_{\mathfrak{M}_{0}}]$. Now, the algebra $\BC[(\mathcal{M}_{\mathfrak{t}_{\mathcal{M}}})^{\nu_{\mathfrak{M}}(\BC^\times)}]$ is finitely generated over $\BC[\mathfrak{s}_{\mathfrak{M}_{0}}]$ iff $\mathcal{M}^{\nu_{\mathfrak{M}}(\BC^\times)}=\mathcal{M}^{S_{\mathcal{M}}}$ (considered as a set) consists of one point. Since we want to identify the algebras above we must make the following assumption.

\begin{Ass} \label{Fixed_points_finite}
The set $\mathcal{M}^{S_{\mathcal{M}}}$ consists of one point.
\end{Ass}

\begin{Rem}
{\emph{Let us recall that symplectic duality predicts that $\mathfrak{M}_{0}$ has a symplectic resolution iff $(\mathfrak{M}_{0}^!)^{S_{\mathfrak{M}_{0}^!}}$ consists of one point.}}
\end{Rem}

The following conjecture is due to Hikita and Nakajima. We will call it {\emph{Hikita-Nakajima conjecture}}. 

\begin{Conj}\label{Hikita_Nakajima_conjecture_general_version}
There is an isomorphism of $\BZ$-graded algebras over $\BC[\mathfrak{s}_{\mathfrak{M}_{0}}]$:
\begin{equation*}
H^*_{S_{\mathfrak{M}_0}}(\mathfrak{M},\BC) \simeq \BC[(\CM_{\mathfrak{t}_{\mathcal{M}}})^{\nu_{\mathfrak{M}}(\BC^\times)}].
\end{equation*}
\end{Conj}

\begin{Rem}\label{quantized_Hikita_Nakajima}
\footnotesize
{\emph{Actually Nakajima's version of this conjecture is even more general: on the LHS, we can consider the algebra $H^*_{S_{\mathfrak{M}_{0}} \times \BC^\times_\hbar}(\mathfrak{M},\BC)$ and on the RHS we can consider the Rees algebra of ``quantized schematic fixed points'' (so-called $B$-algebra or Cartan subquotient) of the  quantization of $\BC[\mathcal{M}_{\mathfrak{t}_{\mathcal{M}}}]$ (see, for example, \cite[Section 5.6]{joel_new}). In \cite{Qhikita}, another generalization of this conjecture is proposed (where one replaces cohomology on the LHS with quantum cohomology).}}
\normalsize
\end{Rem}

\begin{Rem}
{\emph{Let us note that if $\mathfrak{M}$, $\mathfrak{M}'$ are two symplectic resolutions of $\mathfrak{M}_{0}$, then the algebras $H^*_{S_{\mathfrak{M}_{0}}}(\mathfrak{M},\BC)$, $H^*_{S_{\mathfrak{M}_{0}}}(\mathfrak{M}',\BC)$ are isomorphic. This follows from the fact that the universal deformations $\mathfrak{M}_{\mathfrak{t}_{\mathfrak{M}_0}} \ra \mathfrak{t}_{\mathfrak{M}_0}$, $\mathfrak{M}'_{\mathfrak{t}_{\mathfrak{M}_0}} \ra \mathfrak{t}_{\mathfrak{M}_0}$ are locally trivial in $C^\infty$-topology (see \cite[Section 1.2 and references therein]{NAM2}), so $\mathfrak{M}$, $\mathfrak{M}'$ are both diffeomorphic to a generic fiber of $\mathfrak{M}_{0,\mathfrak{t}_{\mathfrak{M}_0}} \ra \mathfrak{t}_{\mathfrak{M}_0}$. We are grateful to Pavel Etingof for explaining this to us. 
Similarly, one can see that for any generic cocharacter $\nu\colon \BC^\times \ra S_{\mathcal{M}}$ the schematic fixed points $\mathcal{M}_{\mathfrak{t}_{\mathcal{M}}}^{\nu(\BC^\times)}$ are the same and are isomorphic to the schematic fixed points $\mathcal{M}_{\mathfrak{t}_{\mathcal{M}}}^{S_{\mathcal{M}}}$. Indeed, note that $\mathcal{M}_{\mathfrak{t}_{\mathcal{M}}}$ can be $S_{\mathcal{M}}$-equivariantly embedded in some vector space $O$ with a linear action of $S_{\mathcal{M}}$. Then  $\mathcal{M}_{\mathfrak{t}_{\mathcal{M}}}^{\nu(\BC^\times)}$, $\mathcal{M}_{\mathfrak{t}_{\mathcal{M}}}^{S_{\mathcal{M}}}$ are (schematic) intersections 
\begin{equation*}
\mathcal{M}_{\mathfrak{t}_{\mathcal{M}}}^{\nu(\BC^\times)}=\mathcal{M}_{\mathfrak{t}_{\mathcal{M}}} \cap O^{\nu(\BC^\times)},~ \mathcal{M}_{\mathfrak{t}_{\mathcal{M}}}^{S_{\mathcal{M}}}= \mathcal{M}_{\mathfrak{t}_{\mathcal{M}}} \cap O^{S_{\mathcal{M}}}.
\end{equation*}
This reduces the claim to showing that $O^{\nu(\BC^\times)}=O^{S_{\mathcal{M}}}$  for a generic $\nu\colon \BC^\times \ra S_{\mathcal{M}}$ that directly follows from Proposition \ref{Y_smooth_implies_fixed_smooth} since $O$ is smooth. We are grateful to Alexander Braverman for explaining this to us.}}
\end{Rem}

\begin{Rem}
{\emph{Let us mention one ``combinatorial'' corollary of Conjecture \ref{Hikita_Nakajima_conjecture_general_version}. 
Pick a point $s \in \mathfrak{s}_{\mathfrak{M}_{0}}$. We have an isomorphism of algebras $H^*_{S_{\mathfrak{M}_{0}}}(\mathfrak{M})|_{s} \simeq H^*(\mathfrak{M}^s)$. Note  that the semisimple quotient of $H^*(\mathfrak{M}^s)$ is the algebra $H^0(\mathfrak{M}^s)$ that is isomorphic to the direct sum of copies of $\BC$, the number of copies is equal to $|\!\on{Comp}(\mathfrak{M}^s)|$, where $\on{Comp}(\mathfrak{M}^s)$ is the set of connected components of $\mathfrak{M}^s$. 
Conjecture \ref{Hikita_Nakajima_conjecture_general_version} implies that the algebra $H^*(\mathfrak{M}^s)$ is isomorphic to the algebra $\BC[\mathcal{M}_s^{\nu_{\mathfrak{M}}(\BC^\times)}]$, where $\mathcal{M}_s$ is the fiber over $s \in \mathfrak{s}_{\mathfrak{M}_{0}}$ of the  deformation $\mathcal{M}_{\mathfrak{s}_{\mathfrak{M}_{0}}} \ra \mathfrak{s}_{\mathfrak{M}_{0}}$. The semisimple quotient of $\BC[\mathcal{M}_s^{\nu_{\mathfrak{M}}(\BC^\times)}]$ is isomorphic to the direct sum of copies of $\BC$ labeled by the set $\mathcal{M}_s^{\nu_{\mathfrak{M}}(\BC^\times)}(\BC)$ of $\BC$-points of $\mathcal{M}_s^{\nu_{\mathfrak{M}}(\BC^\times)}$. So we obtain the bijection of sets:
\begin{equation*}
\on{Comp}(\mathfrak{M}^s) \simeq \mathcal{M}_s^{\nu_{\mathfrak{M}}(\BC^\times)}(\BC).
\end{equation*}
}}
\end{Rem}

\begin{Warning}
Hikita-Nakajima conjecture is not true as stated for arbitrary pairs of symplectically dual varieties. 
A counterexample is a part of work in progress by K. Hoang, D. Matvieievskyi, and the first author.

\end{Warning}
Let us briefly recall the current state of the Hikita-Nakajima conjecture. The Hikita conjecture was proven for the case of Hilbert scheme of points on $\BA^2$, type $A$ Slodowy slices  and hypertoric varieties in \cite[Theorem 1.1, Theorem A.1, Theorem B.1]{HI}.
In \cite{Pasha}, the second author has proven the case of $Y=\C^2/\Gamma$ where $\Gamma$ is a finite subgroup of  $\on{SL}_2(\C)$.
In the paper \cite[Theorem 1.0.5]{HA}, Hatano proved that  $H^*(\mathfrak{M}(n,r),\BC)$ and $\BC[(\mathfrak{M}_0(n,r)^{!})^{\BC^\times}]$ are isomorphic as graded vector spaces. 
In \cite[Theorem 1.5]{KTWWY}, Kamnitzer, Tingley, Webster, Weekes and Yacobi have proven Hikita conjecture for the $ADE$ slices in the affine Grassmanian. They have also proven the equivariant  (see Remark \ref{quantized_Hikita_Nakajima}) version of Conjecture 
\ref{Hikita_Nakajima_conjecture_general_version} for type $A$ quivers and some weaker form of this conjecture for $DE$ quivers (see \cite[Section 8.3]{KTWWY}, \cite[Section 1.3 \text{and} Theorem 1.5]{KTWWY2}, see also \cite[Section 6.6]{joel_new}). 
The results of \cite{Pasha} were recently generalized to the equivariant case in \cite{CWS}.

The main result of this paper is the following theorem. 
\begin{Th} \label{theresult}
 Hikita-Nakajima conjecture holds for
 $\mathfrak{M}=\mathfrak{M}(n,r)$, the Gieseker variety. 
\end{Th}
Actually, we will describe the isomorphism explicitly using certain generators of the algebras mentioned above (see Theorem~\ref{MAIN_THEOREM} for more details).

\begin{Rem} 
{\emph{Let us point out that our approach gives a new proof of the Hikita conjecture for the Hilbert scheme case (when $r=1$), which was proved by Hikita himself. It also generalizes the results of \cite{HA}, where the author proves that  \newline $H^*(\mathfrak{M}(n,r),\BC)$, $\BC[(\mathfrak{M}_0(n,r)^{!})^{\BC^\times}]$ are isomorphic as graded vector spaces (Appendix \ref{fixed_points_are_free_section}  repeats
some of the arguments of the paper \cite{HA}).  The main idea of our paper is that using deformations simplifies the picture.}} \end{Rem}

\subsection{Gieseker variety ($ADHM$ space)}
Gieseker variety $\mathfrak{M}(n,r)$ depends on a pair $n,r \in \BZ_{\geqslant 1}$ of positive integer numbers. It can be realized as a Nakajima quiver variety corresponding to the Jordan quiver (see Definition~\ref{def_gieseker}). Variety $\mathfrak{M}(n,r)$ also has a realization as the moduli space of torsion-free sheaves on $\BP^2$ of rank $r$ with the second Chern class being $n$ and with a fixed trivialization at the line at infinity (see \cite[Chapter 2]{NA0} for details). The variety $\mathfrak{M}(n,r)$ is an  important object that originally came from physics.
The Gieseker variety $\mathfrak{M}(n,r)$ is a resolution of singularities of the variety  $\mathfrak{M}_0(n,r)=\on{Spec}\BC[\mathfrak{M}(n,r)]$. The variety $\mathfrak{M}_0(n,r)$  has a realization as an (affine) Nakajima quiver variety corresponding to the Jordan quiver (see Definition~\ref{def_gieseker}).

So we are taking $\mathfrak{M}=\mathfrak{M}(n,r)$, $\mathfrak{M}_{0}=\mathfrak{M}_0(n,r)$. 
Then the torus $S_{\mathfrak{M}_{0}}$ can be described as follows. We have a natural symplectic action of 
$
\on{SL}_r$ on $\mathfrak{M}(n,r)$ 
(via changing the trivialization at infinity) and also the action of $\BT=\BC^\times$ via the action on $\BP^2$ that multiplies one coordinate by $t$ and another by $t^{-1}$ (so-called ``hyperbolic action''). Let $T_r \subset \on{SL}_r$ be a maximal torus. The torus $S_{\mathfrak{M}_{0}}$ is the image of $A:=\BT \times T_r$ in $\on{Aut}_{\BC^\times_\hbar}(\mathfrak{M}(n,r))$ and $\mathfrak{s}_{\mathfrak{M}_{0}}$ naturally identifies with $\mathfrak{a}:=\on{Lie}A$. The space $\mathfrak{t}_{\mathfrak{M}_0(n,r)}=H^2(\mathfrak{M}(n,r),\BC)$ is known to be one-dimensional.

\subsection{Symplectic dual to $\mathfrak{M}(n,r)$}
Let us now give a very brief description of the variety $\mathcal{M}(n,r)=\mathfrak{M}_0(n,r)^{!}$ and its  deformation $\mathcal{M}(n,r)_{\mathfrak{a}} \ra \mathfrak{a}$.

One way to construct a dual to $\mathfrak{M}_0(n,r)$ is via the Coulomb branches introduced in~\cite{BFN}. In this approach,   
$\mathcal{M}(n,r)$ is equal to the spectrum of the algebra $H_*^{(\on{GL}_n)_{\CO}}(\CR_{n,r})$
of $(\on{GL}_n)_{\CO}$-equivariant
Borel-Moore homology of the variety of triples $\CR_{n,r}$ corresponding to the Jordan quiver (see Section~\ref{dual_via_Coulomb_constr_tripl_hom} for details). The variety $\mathcal{M}(n,r)_{\mathfrak{a}}$ is then the spectrum of $A \times (\on{GL}_n)_{\CO}$ - equivariant homology of the same variety of triples $\mathcal{R}_{n,r}$ (see Section~\ref{two_descriptions)of_sympl_dual_section} for details). 
The Coulomb branch $\mathcal{M}(n,r)_{\mathfrak{a}}$ above  can be realized (see~\cite{KONA}, \cite{BEF}, \cite{webster}) as the spectrum of the center $Z(H_{n,r})$ of the  rational Cherednik algebra $H_{n,r}$ corresponding to the group $\Gamma_n:=S_n \ltimes (\BZ/r\BZ)^n$ (see Section ~\ref{DAHA_corresp_to_Gamma_section} for details). 
Thus we have 
\begin{equation*}
\mathfrak{M}_0(n,r)^{!}_{\mathfrak{a}}=\CM(n,r)_{\mathfrak{a}}=\on{Spec}Z(H_{n,r}). 
\end{equation*}

The algebra $H_{n,r}$ has a natural $\BZ$-grading (see~(\ref{Z_grading_on _H})) that induces the action of $\BT=\BC^\times$ on $\on{Spec}Z(H_{n,r})=\CM(n,r)_{\mathfrak{a}}$. 

\begin{Rem} 
{\emph{The action $\BT \curvearrowright \CM(n,r)_{\mathfrak{a}}$ in the Coulomb terms is described in \cite[Section 3 (v)]{BFN}.}}
\end{Rem}

\begin{Rem}\label{dual_as_quiv_remark}
{\emph{One can also construct $\mathfrak{M}_0(n,r)^{!}$ as the (affine) Nakajima quiver variety $\mathcal{X}_0(n,r)$ corresponding to the cyclic quiver with $r$ vertices labeled by $\BZ/r\BZ$ and having $n$-dimensional vector spaces placed at these vertices and one-dimensional framing at the vertex corresponding to zero. The deformation $\mathcal{X}_0(n,r)_{\mathfrak{a}}$ can be constructed similarly (c.f. Definition \ref{univ_quiver_def} below). The identification $\mathcal{X}_0(n,r)_{\mathfrak{a}} \simeq \on{Spec}Z(H_{n,r})$ is given by the Etingof-Ginzburg isomorphism (that goes back to the paper \cite{EG}). More detailed, in~\cite[Section 3.3]{TO} it is explained (following the proof of~\cite[Theorem 11.16]{EG}) how to construct isomorphisms between smooth fibers of the families $\mathcal{X}_0(n,r)_{\mathfrak{a}}$, $\on{Spec}Z(H_{n,r})$ over ${\mathfrak{a}}$. One can show (we are grateful to Pavel Etingof and Ivan Losev for the explanations) that these isomorphisms extend to the desired isomorphism $\mathcal{X}_0(n,r)_{\mathfrak{a}} \simeq \on{Spec}Z(H_{n,r})$. The idea is to first extend these isomorphisms to the smooth locus over fiber over {\em{every}} $c \in {\mathfrak{a}}$ and then use the normality of our varieties to extend in codimension two.}}
\end{Rem}

\subsection{Cyclotomic Hecke algebra and its center}
Let $Q_{n,r}$ be the algebra of functions on schematic $\BT$-fixed points of $\on{Spec}Z(H_{n,r})$ (that can also be considered as the algebra of functions on $\mathcal{M}(n,r)_{\mathfrak{a}}^{\BT}$). Recall that by Proposition~\ref{shematic_fixed_for_affine}  we have 
\begin{multline*}
Q_{n,r}=Z(H_{n,r})_0/\sum_{i>0}Z(H_{n,r})_{-i} Z(H_{n,r})_{i}=\\
=Z(H_{n,r})/(b \in Z(H_{n,r})_i,\, i \neq 0),
\end{multline*}
where the grading on $H_{n,r}$ is as in (\ref{Z_grading_on _H}):
\begin{equation*}
\on{deg}x_j=1,\, \on{deg}y_j=-1,\, \on{deg}\Gamma_n=\on{deg}{\bf{h}}=0.   \end{equation*}
Our goal is to identify the algebra $Q_{n,r}$ with the algebra $H^*_{A}(\mathfrak{M}(n,r),\BC)$. It turns out that there is another algebra that is isomorphic to both of the algebras above. This algebra is the center $Z(R^r(n))^{JM}$ of the  cyclotomic degenerate Hecke  algebra $R^r(n)$ (see Section \ref{cyclot_degen_defe} for the definition). It was observed in~\cite{SVV} that  algebras $Z(R^r(n))^{JM}$, $H^*_{A}(\mathfrak{M}(n,r),\BC)$ are isomorphic. We give an independent (but certainly similar) proof of this fact.

\subsection{Main idea of the proof}\label{main_idea_of_the_proof}
It turns out that there exists one ``universal''  approach that allows us to identify algebras 
\begin{equation}\label{our_alg_intro}
H^*_{A}(\mathfrak{M}(n,r),\BC),\, Z(R^r(n))^{JM},\,  Q_{n,r} \simeq \BC[\mathcal{M}(n,r)_{\mathfrak{a}}^{\BT}]
\end{equation} 
with each other simultaneously. The idea is simple: we  embed all of the algebras above inside the algebra 
\begin{equation*}
E:=\bigoplus_{\boldsymbol{\la} \in \CP(r,n)} \BC[\mathfrak{a}]=\BC[\mathfrak{a}]^{\oplus |\CP(r,n)|}
\end{equation*}
and show that their images coincide. Here $\CP(r,n)$ is the set of $r$-multipartitions of $n$ (see Definition~\ref{multipart}). In order to show that images are the same we consider natural generators of these algebras and show that their images in $E$ are the same. 
In particular, we obtain explicit descriptions for isomorphisms between algebras in (\ref{our_alg_intro}).

\begin{Rem}
{\emph{Let $F$ be the function field of the parameter space $\mathfrak{a}=\BA^r$. We will see that the embeddings above become isomorphisms after tensoring by $F$ or, more precisely, after localizing at a certain finite set of elements of ${\bf{k}}:=\BC[\mathfrak{a}]$ (certain ``walls''). Compare with Remark \ref{more_car_loc_statement} below.}}
\end{Rem}

\begin{Rem}
{\emph{In \cite[Definition 2.1]{BLPPW} the authors introduce a notion of  {\em{localization algebra}}.  
We observe that algebras that appear  in (\ref{our_alg_intro}) have natural structures of (strong, free) localization algebras and all of them are isomorphic (as algebras with this additional structure).}}
\end{Rem}

\subsubsection{Embedding $Z(R^r(n))^{JM} \subset E$ and  generators of $Z(R^r(n))^{JM}$} For every ${\boldsymbol{\la}} \in \CP(r,n)$ one can consider the corresponding ``universal'' Specht modules over $R^r(n)$ that we denote by $\widetilde{S}_{{\bf{k}}}
(\boldsymbol{\la})$ (see Section \ref{rep_theory_of_cyclotomic} for details). Acting by the center $Z(R^r(n))^{JM}$ of $R^r(n)$ on these modules we obtain the desired embedding 
\begin{equation*}
\psi\colon Z(R^r(n))^{JM} \subset \bigoplus_{\boldsymbol{\la} \in \CP(r,n)} \on{End}_{R^r(n)}(\widetilde{S}_{\bf{k}}(\boldsymbol{\la}))=E.   
\end{equation*}
It follows from~\cite[Theorem 1]{Br} that natural generators of $Z(R^r(n))^{JM}$ are classes of elements 
\begin{equation*}
e_k(z_1,\ldots,z_n),\, k=1,\ldots,n,
\end{equation*}
where $e_k \in \BC[z_1,\ldots,z_n]^{S_n}$ are elementary symmetric polynomials.

\subsubsection{Embedding $H^*_{A}(\mathfrak{M}(n,r),\BC) \subset E$ and generators of $H^*_{A}(\mathfrak{M}(n,r),\BC)$}
The set of $A$-fixed points $\mathfrak{M}(n,r)$ can be parametrized by $\CP(r,n)$ (see Section \ref{equivariant_cohomology_embedding} for details). 

In particular, we have the natural embedding 
\begin{equation*}
\iota\colon \mathfrak{M}(n,r)^{A} \subset \mathfrak{M}(n,r) 
\end{equation*}
that induces the desired embedding \begin{equation*}
\iota^*\colon H^*_{A}(\mathfrak{M}(n,r),\BC) \subset H^*_{A}(\mathfrak{M}(n,r)^{A},\BC)=E.   
\end{equation*}
It remains to note that the  algebra  $H^*_{A}(\mathfrak{M}(n,r),\BC)$ is generated by the elements 
\begin{equation*}
c_k(\mathcal{V}),\, k=1,\ldots,n, 
\end{equation*}
where $c_k(\bullet)$ is the $k$th $A$-equivariant Chern class and $\mathcal{V}$ is the tautological $n$-dimensional vector bundle on $\mathfrak{M}(n,r)$. This, for example, follows from~\cite[Corollary 1.5]{MN}.

\subsubsection{Embedding $Q_{n,r} \subset E$ and generators of $Q_{n,r}$}
Recall that $Q_{n,r}$ is a quotient of $Z(H_{n,r}) \subset H_{n,r}$. To every ${\boldsymbol{\la}} \in \CP(r,n)$ one can associate the corresponding induced (graded) $H_{n,r}$-module $\Delta({\boldsymbol{\la}})$ on which $Z(H_{n,r})$ will act by some character and this action factors through the action of $Q_{n,r}$ (see Section \ref{standard_simple_RCA}). This gives us the desired embedding \begin{equation*}
\phi \colon Q_{n,r} \subset \bigoplus_{\boldsymbol{\la} \in \CP(r,n)}\on{End}^{\mathrm{gr}}_{H_{n,r}}(\Delta(\boldsymbol{\la}))=E.
\end{equation*}
Generators of $Q_{n,r}$ are classes of
\begin{equation*}
e_k(u_1,\ldots,u_n),\, k=1,\ldots,n,
\end{equation*}
where $u_i \in H_{n,r}$ are Dunkl-Opdam elements (see Lemma~\ref{gen_of_Q!}).

\begin{Rem}
{\emph{The identification $Z(H_{n,r}) \simeq H_*^{A \times (\on{GL}_n)_{\CO}}(\CR_{n,r})$ sends \newline $e_k(u_1,\ldots,u_n)$ to the function $m_k := c_k(V) * 1$, here $c_k(V) \in H^*_{A \times \on{GL}_n}(\on{pt})$ is the Chern class of the tautological $A \times \on{GL}_n$-bundle $V=\BC^n$ on $\on{pt}$ and $*$ is the convolution product on $H_*^{A \times (\on{GL}_n)_{\CO}}(\CR_{n,r})$ (see \cite[Section 4]{webster} for details). This should be compared with the fact  that generators on the dual side (i.e., generators  of $H^*_A(\mathfrak{M}(n,r),\BC)$) are the Chern classes $c_i(\CV)$ of the tautological bundle $\mathcal{V}$.}}
\end{Rem}

\begin{Rem}
{\emph
The embedding $\phi$ can be interpreted geometrically as a pullback homomorphism from schematic fixed points on $\mathcal{M}(n,r)_{\mathfrak{a}}$ to schematic fixed points on a resolution of $\mathcal{M}(n,r)_{\mathfrak{a}}$ (see Section \ref{Hikita_Hilbert} for the $r=1$ example and Section \ref{general_section!} for general conjectures in this direction).}

\end{Rem}

\subsection{Main results and structure of the paper}

\subsubsection{Identification of the parameters and the main Theorem}
We have the parameter spaces 
\begin{equation}\label{all_param_spaces0}
{\bf{h}}=\BC[\kappa,c_1,\ldots,c_{r-1}],\, {\bf{k}}=\BC[\kappa,a_1,\ldots,a_r]/(a_1+\ldots+a_r)=\BC[\mathfrak{a}],
\end{equation}
used in definitions of  algebras (\ref{our_alg_intro}) that we want to identify.
We identify the parameters 
as follows ($i=1,\ldots,r$, $k=1,\ldots,r-1$):
\begin{equation}\label{identify_all_parameters}
a_i=p(\eta^{i-1}),
\end{equation}
where $p(q)= \frac{1}{r}\sum_{l=1}^{r-1}\frac{c_l}{\eta^{-l}-1}q^l \in \BC[q]$, $\eta=e^{\frac{2\pi \sqrt{-1}}{r}}$. We will denote by $F$ the field of fractions of the algebras (\ref{all_param_spaces0}).

\begin{Th}\label{MAIN_THEOREM}
After the identification of parameters ~(\ref{identify_all_parameters}) the $\BZ$-graded algebras 

\begin{equation*}
H^*_A(\mathfrak{M}(n,r),\BC),\, Z(R^r(n))^{JM},\, Q_{n,r},\, \BC[\on{Spec}(H^{A \times (\on{GL}_n)_{\CO}}_*(\CR_{n,r}))^{\mathbb{T}}]
  \end{equation*}
are isomorphic. The isomorphisms above identify generators as follows 

\begin{equation}\label{identifications_formulas}
c_k(\mathcal{V})=[e_k(z_1,\ldots,z_n)]=[e_k(u_1,\ldots,u_n)]=[m_k],
  \end{equation}
where $\CV$ is the tautological rank $n$ vector bundle on $\mathfrak{M}(n,r)$, $m_k=c_k(V) * 1$ and $u_i \in H_{n,r}$ are Dunkl-Opdam elements (Definition (\ref{OD_elements_DEF!})).
\end{Th}

\begin{Rem}
{\emph{Note that the isomorphisms above are automatically graded since they preserve the degree of the generators (\ref{identifications_formulas}). Note that the isomorphism $$H^*_A(\mathfrak{M}(n,r),\BC) \simeq \BC[\on{Spec}(H^{A \times (\on{GL}_n)_{\CO}}_*(\CR_{n,r}))^{\mathbb{T}}]$$ is already an isomorphism of $\BC[\kappa,a_1,\ldots,a_n]/(a_1+\ldots+a_n)=\BC[\mathfrak{a}]$-algebras and there is no need to identify parameters.}}
\end{Rem}

\subsubsection{Structure of the paper} The paper is organized as follows. In Section \ref{Hikita_Hilbert}, we prove the Hikita-Nakajima conjecture for $r=1$, i.e., for the Hilbert scheme $\on{Hilb}_n(\BA^2)$ using results of~\cite{VA}. 
In Section \ref{two_descriptions)of_sympl_dual_section}, we consider the case of arbitrary $r$ and describe symplectically dual variety to $\mathfrak{M}_0(n,r)$,  along with its deformation using Coulomb branches and rational Cherednik algebras.

In Sections ~\ref{equivariant_cohomology_embedding}, \ref{rep_theory_of_cyclotomic} and \ref{schem_fixed_Q_section} we provide the realization of the idea behind the proof of Hikita-Nakajima conjecture that we briefly introduced in Section~\ref{main_idea_of_the_proof}.
First, in Section~\ref{equivariant_cohomology_embedding} we describe the embedding $H^*_A(\mathfrak{M}(n,r),\BC) \subset E$ and determine the image of the generators $c_i(\mathcal{V}) \in H^*_A(\mathfrak{M}(n,r),\BC)$ under this embedding.
Then, in Section \ref{rep_theory_of_cyclotomic} we define the cyclotomic degenerate Hecke algebra $R^r(n)$ and review its representation theory. Using this, we describe the embedding $Z(R^r(n))^{JM} \subset E$ and determine its image. 
Next, in Section \ref{schem_fixed_Q_section}, we recall the representation theory of the rational Cherednik algebra $H_{n,r}$ and  then describe the embedding $Q_{n,r} \subset E$ (using the representation theory of $H_{n,r}$). In Lemma~\ref{gen_of_Q!}, we describe generators of $Q_{n,r}$ and determine their images under the embedding $Q_{n,r} \subset E$.
Finally, as a corollary of the  results of Sections ~\ref{equivariant_cohomology_embedding}, \ref{rep_theory_of_cyclotomic}, \ref{schem_fixed_Q_section}, we  obtain Theorem \ref{MAIN_THEOREM} (see Theorem \ref{prf_of_main_th_arg} in Section \ref{proof_of_main_section}). 
In Section~\ref{general_section!}, we
discuss a potential approach to proving the Hikita-Nakajima conjecture for more general quivers. Appendix \ref{fixed_points_are_free_section} contains a (short) proof of the fact that $Q_{n,r} \simeq \BC[\CM(n,r)_{\mathfrak{a}}^{\BT}]$ is flat over the space of parameters.

\subsection{Acknowledments}
 We  express our deep gratitude to Roman Bezrukavnikov, Pavel Etingof, Joel Kamnitzer, Dima Korb, Ivan Losev, Tomek Przezdziecki, Ben Webster,  and Alex Weekes   for stimulating conversations, useful remarks, and generalization suggestions.

 Special thanks go to Ben Webster for organizing a learning seminar on Coulomb branches and making it possible for the first author to give a talk on it.

\section{Hikita-Nakajima conjecture for Hilbert scheme} \label{Hikita_Hilbert}
In this section, we provide an ``elementary'' proof of the  Hikita-Nakajima conjecture for the Hilbert scheme of points on $\BA^2$. This proof approach differs from the original method used by Hikita (see~\cite{HI}) As a corollary, when we set the equivariant parameter to zero, we obtain a theorem that was originally proved by Hikita.

The idea is the following: we identify the algebra of schematic fixed points with the Rees algebra of the center $Z(\BC S_n)$ of $\BC S_n$ and then use the identification 

\begin{equation}\label{vass_iso}
\on{Rees}(Z(\BC S_n)) \simeq H_{\BT}^*(\on{Hilb}_n(\BA^2))~\text{ (see \cite{VA})}
\end{equation}
 to obtain the Hikita-Nakajima conjecture for $\on{Hilb}_n(\BA^2)$. An alternative approach to the identification
(\ref{vass_iso}) appears in Sections \ref{equivariant_cohomology_embedding}, \ref{rep_theory_of_cyclotomic}.

\subsection{Nakajima quiver varieties}\label{naka_quiver_def} Let $I=(I_0,I_1)$ be a finite quiver, where $I_0$ is the set of vertices, and $I_1$ is the set of oriented edges.
Let ${\bf{v}}=(v_i)_{i \in I_0}$, ${\bf{w}}=(w_i)_{i \in I_0}$ be $I_0$-tuples of nonnegative integer numbers. 
Let $V=\bigoplus_{i \in I_0}V_i$, $W=\bigoplus_{i \in I_0}W_i$ be  $I_0$-graded vector spaces with $\on{dim}V_i=v_i$, $\on{dim}W_i=w_i$.
Consider the representation space
\begin{equation*}
{\bf{N}}={\bf{N}}(\vv,\ww)={\bf{N}}_I({\bf{v}},{\bf{w}}):= \bigoplus_{(i \ra j) \in I_1}\on{Hom}(V_i,V_j) \oplus \bigoplus_{i \in I_0}\on{Hom}(W_i,V_i).    
\end{equation*}
We also consider the cotangent space $\M_I(\vv,\ww)=\M(\vv,\ww)={\bf{M}}:=T^*{\bf{N}}$ that can be identified with 
\begin{multline*}
\bigoplus_{(i \ra j) \in I_1}\on{Hom}(V_i,V_j) \oplus \bigoplus_{(i \ra j) \in I_1}\on{Hom}(V_j,V_i) \oplus  \\
\oplus  \bigoplus_{i \in I_0}\on{Hom}(W_i,V_i) \oplus \bigoplus_{i \in I_0}\on{Hom}(V_i,W_i).
\end{multline*}
We can represent elements of $\M(\vv,\ww)$ as quadruples $(X,Y,\gamma,\delta)$, where 
\begin{equation*}
X \in    \bigoplus_{(i \ra j) \in I_1}\on{Hom}(V_i,V_j),\, Y \in  \bigoplus_{(i \ra j) \in I_1}\on{Hom}(V_j,V_i), 
\end{equation*}
\begin{equation*}
\gamma \in \bigoplus_{i \in I_0}\on{Hom}(W_i,V_i),\, \delta \in \bigoplus_{i \in I_0}\on{Hom}(V_i,W_i).    
\end{equation*}
The space $\M(\vv,\ww)=T^*{\bf{N}}$ carries a natural symplectic form.
We set
\begin{equation*}
G_{\vv}:=\prod_{i \in I_0} \on{GL}(V_i),\, \mathfrak{g}_{\vv}:=\bigoplus_{i \in I_0}\mathfrak{gl}(V_i).
\end{equation*}
The group $G_{\vv}$ acts naturally on the vector space $\M(\vv,\ww)$. This action is symplectic with the moment map:
\begin{equation*}
\mu\colon \M(\vv,\ww) \ra \mathfrak{g}_{\vv},\, (X,Y,\gamma,\delta) \mapsto [X,Y]+\gamma\delta. 
\end{equation*}

\begin{Def}
A quadruple $(X,Y,\gamma,\delta) \in \M(\vv,\ww)$ is called {\em{stable}} if for every $X, Y$-invariant graded subspace $S \subset V$ such that $S$ contains $\on{im}\gamma$ we have $S=V$. 
We denote by $\M(\vv,\ww)^{\mathrm{st}} \subset \M(\vv,\ww)$ the (open) subset of stable quadruples.
\end{Def}

\begin{Def}
The Nakajima quiver varieties $\mathfrak{M}(\vv,\ww)$, $\mathfrak{M}_0(\vv,\ww)$ are defined as the following quotients: 
\begin{equation*}
\mathfrak{M}(\vv,\ww):=\mu^{-1}(0)^{\mathrm{st}}/G_{\vv},\, \mathfrak{M}_0(\vv,\ww):=\mu^{-1}(0)/\!\!/G_{\vv}.
\end{equation*} 
\end{Def}
 We have the natural (projective) morphism $\pi\colon \mathfrak{M}(\vv,\ww) \ra \mathfrak{M}_0(\vv,\ww)$. 

\begin{Ass}
We assume that $\pi$ is a resolution of singularities.
\end{Ass}

Let $\mathfrak{z}_{\vv} \subset \mathfrak{g}_{\vv}$ be the center of $\mathfrak{g}_{\vv}$.
The varieties $\mathfrak{M}(\vv,\ww)$, $\mathfrak{M}_0(\vv,\ww)$ admit certain natural deformations over the space $\mathfrak{z}_{\vv}$.

\begin{Def}\label{univ_quiver_def}
The ``universal'' quiver varieties $\mathfrak{M}(\vv,\ww)_{\mathfrak{z}_{\vv}}$, $\mathfrak{M}_0(\vv,\ww)_{{\mathfrak{z}_{\vv}}}$ are defined as follows:
\begin{equation*}
\mathfrak{M}(\vv,\ww)_{{\mathfrak{z}_{\vv}}}:=\mu^{-1}(\mathfrak{z}_{\vv})^{\mathrm{st}}/G_{\vv},\,    \mathfrak{M}_0(\vv,\ww)_{{\mathfrak{z}_{\vv}}}:=\mu^{-1}(\mathfrak{z}_{\vv})/\!\!/G_{\vv}. 
\end{equation*}

For ${\bf{a}} \in \mathfrak{z}_{\vv}$, we denote by $\mathfrak{M}(\vv,\ww)_{{\bf{a}}}$, $\mathfrak{M}_0(\vv,\ww)_{\bf{a}}$ the fibers of these families over ${\bf{a}}$.
\end{Def}

\subsection{Hilbert scheme $\on{Hilb}_n(\BA^2)$ and $S^n(\BA^2)$ as Nakajima quiver varieties}

Our main object of study in this section is $\on{Hilb}_n(\BA^2)$ the Hilbert scheme of $n$ points on $\BA^2$. 

\begin{Def}
The variety
$\on{Hilb}_n(\BA^2)$ is the variety whose $\BC$-points are ideals $J \subset \BC[x,y]$ of codimension $n$. The (affine) variety $S^n(\BA^2)$ is the categorical quotient $(\BA^2)^{n}/S_n$.
\end{Def}

Recall that we have the Hilbert-Chow morphism: 
\begin{equation*}
\on{Hilb}_n(\BC^2) \ra S^n(\BA^2),\, J\mapsto \on{Supp}(\BC[x,y]/J).
\end{equation*}
This morphism is a symplectic resolution of singularities.
Let us now recall the description of $\on{Hilb}_n(\BC^2)$ as a Nakajima quiver variety (corresponding to the Jordan quiver).

\begin{Def}\label{def_gieseker}
We denote by $\mathfrak{M}(n,r)$, $\mathfrak{M}_0(n,r)$ the Nakajima quiver varieties corresponding to the quiver $I$ consisting of one vertex and one loop with $\on{dim}V=n$, $\on{dim}W=r$.
\end{Def}

The following proposition holds by \cite[Theorem 2.1 and Proposition 2.9]{NA0}.
\begin{Prop}\label{gieseker_via_hilbert_prop}
There exist isomorphisms 
\begin{equation}\label{Hilb_as_quiver}
\mathfrak{M}(n,1) \iso \on{Hilb}_n(\BA^2),\, \mathfrak{M}_0(n,1) \iso S^n(\BA^2) 
\end{equation}
compatible with natural morphisms $\mathfrak{M}(n,1) \ra \mathfrak{M}_0(n,1)$, $\on{Hilb}_n(\BA^2) \ra S^n(\BA^2)$.
\end{Prop}

We conclude that the points of $\on{Hilb}_n(\BA^2)$, $S^n(\BA^2)$ can be represented as certain quadruples $(X,Y,\gamma,\delta)$ that can be considered as  representations of the following quiver 
(here we identify $V=\BC^n$, $W=\BC$):
\[
\begin{tikzcd}
 \C \arrow[d, shift right=0.5ex,swap, "\gamma"] \\
 \C^n \arrow[u, shift right=0.5ex, swap, "\delta"] 
 \arrow[out=-15,in=15,loop, swap, "Y" ]
 \arrow[out=195,in=165,loop, "X"] 
\end{tikzcd}
\]

\subsection{Calogero-Moser space, deformations of $\on{Hilb}_n(\BA^2)$ and torus actions} We see that varieties $\mathfrak{M}(n,1)_{\mathfrak{z}_n}$, $\mathfrak{M}_0(n,1)_{\mathfrak{z}_n}$ (see Definition~\ref{univ_quiver_def}) are one-parameter deformations of $\on{Hilb}_n(\BA^2)$ and $S^n(\BA^2)$, where the base of the deformation is the center $\mathfrak{z}_n \subset \mathfrak{gl}(V)$ that can be identified with $\BA^1$ via the map $\BA^1 \ni {\bf{a}} \mapsto  {\bf{a}} \cdot \on{Id}_V \in \mathfrak{gl}(V)$.

Let us now discuss  torus actions. Let $\BT$, $\BC^\times_{\hbar}$ be copies of $\BC^\times$. We have an action of $\BT \times \BC^\times_{\hbar}$ on $\on{Hilb}_{n}(\BA^2)$, $S^n(\BA^2)$ that is induced by the action $\BT \times \BC^\times_{\hbar} \curvearrowright \BA^2$ given by 
\begin{equation*}
(t,\hbar) \cdot (x,y)=(t\hbar^{-1} x,t^{-1}\hbar^{-1} y), \, t \in \BT,\, \hbar \in \BC^\times_{\hbar}. \end{equation*}

After  identifications (\ref{Hilb_as_quiver}) the action of  $\BT \times \BC^\times_{\hbar} $ can be described as follows: it is induced from the following action on $\mathfrak{M}(n,1)$:
\begin{equation}\label{action_on_thick_M}
(t,\hbar) \cdot (X,Y,\gamma,\delta)=(t\hbar^{-1} X,t^{-1}\hbar^{-1} Y, \hbar^{-1} \gamma, \hbar^{-1}\delta).
\end{equation}

\begin{Rem}
\emph{Note that $\BT$ acts symplectically, while $\BC^\times_{\hbar}$ scales the symplectic form with  weight $2$.}
\end{Rem}

Formula (\ref{action_on_thick_M}) induces actions 
\begin{equation*}
 \begin{gathered}
\BT \times \BC_{\hbar}^\times \curvearrowright \mathfrak{M}(n,1)_{\mathfrak{z}_n},\\
\BT \times \BC_{\hbar}^\times \curvearrowright \mathfrak{M}_0(n,1)_{\mathfrak{z}_n}.
 \end{gathered}
\end{equation*}
Consider ${\bf{a}} \in \BC^\times$.  Let us describe the fibers $\mathfrak{M}(n,1)_{{\bf{a}}}$, $\mathfrak{M}(n,1)_{{\bf{a}}}$ of the families $\mathfrak{M}(n,1)_{\mathfrak{z}_n}$, $\mathfrak{M}(n,1)_{\mathfrak{z}_n}$ over ${\bf{a}}$. First of all, note that the action of $\BC_{\hbar}$ induces identifications
\begin{equation*}
 \begin{gathered}
  \mathfrak{M}(n,1)_{{\bf{a}}} \simeq \mathfrak{M}(n,1)_1,\,\\ \mathfrak{M}_0(n,1)_{\bf{a}} \simeq  \mathfrak{M}_0(n,1)_1.   \end{gathered}
\end{equation*}

\begin{Def}
Recall that $V$ is a vector space of dimension $n$.
We define Calogero-Moser variety $\CC(n)$ as the following quotient:
\begin{equation*}
\CC(n):=\{(X,Y) \in \on{End}(V)^{\oplus 2}\,|\, \on{rk}\Big([X,Y]-\on{Id}_V\Big)=1\}/\on{GL}(V).    
\end{equation*}
\end{Def}

The following proposition will be useful for us (see, for example,~\cite[Section 1]{WI}):
\begin{Prop}
Natural morphisms 
\begin{equation*}
\mathfrak{M}(n,1)_1 \ra \mathfrak{M}_0(n,1)_1 \ra \mathcal{C}(n),
\end{equation*} given by
\begin{equation*}
[(X,Y,\gamma,\delta)] \mapsto [X,Y,\gamma,\delta] \mapsto [(X,Y)]   
\end{equation*}are isomorphisms. 
\end{Prop}
Thus families $\mathfrak{M}(n,1)_{\mathfrak{z}_n}$, $\mathfrak{M}_0(n,1)_{\mathfrak{z}_n}$ are $\BC^\times_\hbar$-equivariant deformations of 
 \newline $\on{Hilb}_n(\BA^2)$, $S^n(\BA^2)$ over $\BA^1.$ Over a non-zero parameter their fibers are isomorphic to the Calogero-Moser variety $\mathcal{C}(n).$

\subsection{
Hikita-Nakajima conjecture for $\on{Hilb}_n(\BA^2)$}
We denote by $H^*_{\BT}(\on{Hilb}_n(\BA^2))$ the $\BT$-equivariant cohomology of $\on{Hilb}_n(\BA^2)$. This is a $\BZ$-graded algebra over $H^*_{\BT}(\on{pt})$ isomorphic to $\BC[\on{Lie}\BT]$.
We are now ready to state the ``deformed'' version of the Hikita conjecture for $\on{Hilb}_n(\BA^2)$ that we refer to as the Hikita-Nakajima conjecture (see Conjecture \ref{Hikita_Nakajima_conjecture_general_version} above).

\begin{Th}\label{HN_equiv_Hilb}
We have an isomorphism of $\BZ$-graded algebras over $\BC[\mathfrak{z}_{n}] \simeq \BC[\on{Lie}\BT]$ (the identification induced by the isomorphism $\mathfrak{z}_n \simeq \BA^1 \simeq \on{Lie}\BT$):
\begin{equation*}
\BC\Big[\mathfrak{M}_0(n,1)_{\mathfrak{z}_n}^{\BT}\Big] \simeq H_{\BT}^*(\on{Hilb}_n(\BA^2)).    
\end{equation*}
\end{Th}
Our goal is to prove this theorem. We will do it by showing that both algebras in the theorem are isomorphic to the Rees algebra of the center $Z(\BC S_n)$ of the group algebra of $S_n$.

\subsection{Equivariant cohomology of $\on{Hilb}_n(\BA^2)$ and the center of $\BC S_n$}

Let $Z(\BC S_n)$ in $ \BC S_n$ be the center of the group algebra $\BC S_n$. Consider the grading on the vector space $\BC S_n$ defined in the following way: pick a permutation $\sigma \in S_n$ and let $\ell(\sigma)$ be the number of cycles in the decomposition of $\sigma$ as a product of disjoint cycles. We then define 
\begin{equation*}
\on{deg}\sigma:=2(n-\ell(\sigma)).  
\end{equation*}
The grading above induces the increasing $\BZ_{\geqslant 0}$-filtration on $Z_n$:
\begin{multline}\label{filtr_on_center}
\BC=F_0Z_n = F_1Z_n \subset F_2Z_n=F_3Z_n \subset \ldots \\
\ldots \subset Z_n=F_{2n-2}Z_n=F_{2n-1}Z_n=\ldots     
\end{multline}
We denote by $\on{Rees}(Z(\BC S_n))$ the Rees algebra corresponding to the filtration~(\ref{filtr_on_center}). Recall that the algebra $\on{Rees}(Z(\BC S_n))$ is defined as follows:
\begin{equation*}
\on{Rees}(Z(\BC S_n)):=\bigoplus_{m \geqslant 0}\kappa^mF_{2m}Z(\BC S_n) \subset Z(\BC S_n)[\kappa],    
\end{equation*}
where $\kappa$ is a formal parameter of degree $2$.
We consider $\on{Rees}(Z(\BC S_n))$ as an algebra over $\BC[\BA^1]=\BC[\kappa]$.
The following result holds by~\cite{VA} or Sections~\ref{equivariant_cohomology_embedding}, \ref{rep_theory_of_cyclotomic} (see also ~\cite[Theorem 4.7 and Corollary 4.8]{SVV}).
\begin{Prop}\label{cohom_via_center_Hilb}
There is an isomorphism of $\BZ$-graded algebras over $\BC[\on{Lie}\BT] \simeq \BC[\BA^1]$ (the identification is induced by the isomorphism $\on{Lie}\BT \simeq \BA^1$): 
\begin{equation*}
H^*_{\BT}(\on{Hilb}_n(\BA^2)) \simeq \on{Rees}(Z(\BC S_n)).    
\end{equation*}
\end{Prop}

\begin{Rem}
{\emph{Proposition \ref{cohom_via_center_Hilb} can also be proved using the same argument as we use in the proof of Theorem \ref{rees_Z_via_schem_fixed} below.}}
\end{Rem}

Let us now recall the description of the center $Z(\BC S_n)$. To every $k \in 1,\ldots, n$ we can associate the corresponding Jucys–Murphy element $JM_k$ defined as follows:
\begin{equation*}
JM_k:=(1\,k)+(2\,k)+\ldots+(k-1\,k) \in \BC S_n,
\end{equation*}
where $(i\,k) \in S_n$ is the transposition switching $i,k$.
\begin{Rem}
{\emph{Note that $JM_1=0$.}}
\end{Rem}

The following proposition is classical (see, for example,~\cite[Theorem 1.9]{MU}).
\begin{Prop}
The center $Z(\BC S_n)$ is generated (as a vector space) by the elements

$f(JM_1,\ldots,JM_n)$,
where $f$ runs through the symmetric functions on $n$ variables.
\end{Prop}

\subsection{Construction of the isomorphism between schematic fixed points of $\mathfrak{M}_0(n,1)_{\mathfrak{z}_n}$ and $\on{Rees}(Z(\BC S_n))$}\label{sect_idea_iso_schematic_center}
We will prove the following theorem and, using Proposition~\ref{cohom_via_center_Hilb}, obtain Theorem~\ref{HN_equiv_Hilb} as a corollary. 
\begin{Th}\label{rees_Z_via_schem_fixed}
There is an isomorphism of $\BZ$-graded algebras over $\BC[\BA^1]=\BC[\on{Lie}\BT]$
\begin{equation}\label{rees_fixed_iso_itself}
\on{Rees}(Z(\BC S_n)) \iso \BC\Big[\mathfrak{M}_0(n,1)_{\mathfrak{z}_n}^{\BT}\Big]
\end{equation}
that sends $f(JM_1,\ldots,JM_n) \in Z(\BC Z_n)$ to the restriction of the function
\begin{equation*}
\mathfrak{M}_0(n,1)_{\mathfrak{z}_n} \ni [(X,Y,\gamma,\delta)] \mapsto f(\al_1,\ldots,\al_n)
\end{equation*}
to $\mathfrak{M}_0(n,1)_{\mathfrak{z}_n}^{\BT}$.
Here $f$ is a symmetric function on $n$ variables and $\al_1,\ldots,\al_n$ are roots of the characteristic polynomial of $YX \in \on{End}(V)$, the $\BZ$-grading on the $LHS$ of (\ref{rees_fixed_iso_itself}) is the natural grading on $\on{Rees}(Z(\BC S_n))$ and the $\BZ$-grading on the $RHS$ is the one induced by the action of $\BC^\times_\hbar$. 
\end{Th}

The rest of the section is devoted to describing the idea of the proof of Theorem~\ref{rees_Z_via_schem_fixed}. We start with the following proposition, the proof of which is given in Appendix~\ref{fixed_points_are_free_section}.
\begin{Prop}
The algebra $\BC\Big[\mathfrak{M}_0(n,1)_{\mathfrak{z}_n}^{\BT}\Big]$ is flat (hence, free) over $\mathfrak{z}_n=\BA^1$. In particular, we have an isomorphism of $\BZ$-graded algebras 
\begin{equation*}
\BC\Big[\mathfrak{M}_0(n,1)_{\mathfrak{z}_n}^{\BT}\Big] \simeq \on{Rees}(\BC[\mathfrak{M}_0(n,1)_1^{\BT}])=\on{Rees}(\BC[\mathcal{C}(n)^{\BT}]).
\end{equation*}
\end{Prop}

We conclude that  to prove Theorem~\ref{rees_Z_via_schem_fixed} it is enough to construct the isomorphism of filtered algebras $\BC[\CC(n)^{\BT}] \simeq Z(\BC S_n)$. In order to do so we need to describe the algebra $\BC[\CC(n)^{\BT}]$. Let us note that the variety $\CC(n)$ is smooth, hence, the scheme $\CC(n)^{\BT}$ is also smooth (see Proposition~\ref{Y_smooth_implies_fixed_smooth}) and, in particular, reduced. 

The description of the set of fixed points $\CC(n)^{\BT}$ was given by Wilson in \cite[ Proposition 6.11]{WI}, we recall it in Section~\ref{descr_of_fixed_on_CM}. The set $\CC(n)^{\BT}$ is finite and can be parametrized by the set $\CP(n)$ of partitions of $n$, we denote by $[(X^\la,Y^\la)]$ the fixed point corresponding to $\la \in \CP(n)$ (see Definition~\ref{def_X_Y_la}). Since every finite reduced scheme over $\BC$ is just the spectrum of the direct sum of copies of $\BC$, we must have 
\begin{equation}\label{iso_fixed_charac}
\BC[\CC(n)^{\BT}] = \bigoplus_{\la \in \CP(n)}\BC \chi_\la,   
\end{equation}
where $\chi_\la \in \BC[\CC(n)^{\BT}]$ is the characteristic function of the $\BT$-fixed point $[(X^\la,Y^\la)]$ corresponding to $\la \in \CP(n)$.
Recall now that we have the natural identification 
\begin{equation}\label{iso_center_idempot}
Z(\BC S_n) = \bigoplus_{\la \in \CP(n)}\BC {\bf{e}}_\la,
\end{equation} where ${\bf{e}}_\la \in Z(\BC S_n)$ is the idempotent corresponding to the Specht module $S(\la)$ (in other words, for $\nu \in \CP(n)$ the element ${\bf{e}}_\la \in Z(\BC S_n)$ acts on $S(\nu)$ via $\delta_{\la\,\nu}\cdot \on{Id}_{S(\nu)}$).

Composing  (\ref{iso_fixed_charac}) and (\ref{iso_center_idempot}), we obtain the isomorphism of algebras:
\begin{equation*}
\Theta \colon Z(\BC S_n) \iso \BC[\CC(n)^{\BT}],\, {\bf{e}}_\la \mapsto \chi_{\la}. \end{equation*}

To prove Theorem~\ref{rees_Z_via_schem_fixed} it remains to show that the isomorphism $\Theta$ is the one that we need, i.e., it sends element  $f(JM_1,\ldots,JM_n)$ to the restriction of the function $\CC(n) \ni [(X,Y)] \mapsto f(\al_1,\ldots,\al_n)$ to $\CC(n)^{\BT}$. From this we would be able to conclude that the isomorphism $\Theta$ is filtration-preserving.

To prove that $\Theta$ sends element  $f(JM_1,\ldots,JM_n)$ to the function $\CC(n)^{\BT} \ni [(X,Y)] \mapsto f(\al_1,\ldots,\al_n)$ we just need to show that for every symmetric function $f$ on $n$ variables we have 
\begin{equation*}
f(JM_1,\ldots,JM_n)|_{S(\la)}=f(\al_1,\ldots,\al_n) \on{Id}_{S(\la)},
\end{equation*}
where $\al_1,\ldots,\al_n$ is the multiset of  eigenvalues  of $Y^\la X^\la$. 
Recall that \newline
$f(JM_1,\ldots,JM_n)$ acts on $S(\la)$ via the multiplication by $f(c_1,\ldots,c_n)$, where $c_1,\ldots,c_n$ is the multiset of contents of boxes of the Young diagram $\BY(\la)$ corresponding to $\la$ (see (\ref{def_Y_to_la})). It remains to check that the multiset of eigenvalues of $Y^\la X^\la$ is the same as the multiset of contents of boxes of $\BY(\la)$.

\subsection{Description of $\CC(n)^{\BT}$ and eigenvalues of $Y^\la X^\la$}\label{descr_of_fixed_on_CM}

The parametrization of $\CC(n)^{\BT}$ by the elements of $\CP(n)$ goes as follows
(the description was obtained in~\cite{WI}, we follow \cite[Section 5]{TO}).
Pick $m \in \BZ_{\geqslant 1}$ and $1 \leqslant k \leqslant m$.
\begin{Def}
 By $D_m$ we will denote the $m \times m$ matrix with $1$'s on the first diagonal and $0$'s elsewhere, i.e., $D_m=\sum_{i=1}^{m-1}E_{i\, i+1}$. Now, let $Y(m,k)$ be the $m \times m$ matrix such that its only non-zero entries are on the $-1$st diagonal, and it satisfies the relation $\left [Y(m,k), D_m \right ] = mE_{k\,k}$. In other words, the numbers below the diagonal are $1,2, \ldots , k-1, -m+k, \ldots, -2, -1$:
 \begin{equation*}
Y(m,k)= \begin{pmatrix}
0 & 0& \hdotsfor{5} & 0 \\
1& 0 & \hdotsfor{5} & 0\\
0& 2 & 0&\hdotsfor{4} & 0\\
\vdots& \ldots & \ddots & 0 & \hdotsfor{3}& 0 \\
\hdotsfor{3} & k-1 &0& \hdotsfor{3} \\
\hdotsfor{4} & -m+k & 0& \hdotsfor{2} \\
\hdotsfor{8} \\
\hdotsfor{6} & -1 &0  
\end{pmatrix}.
\end{equation*}
\end{Def}

Pick $\la=(\la_1,\ldots,\la_l) \in \CP(n)$. Following \cite[Section 4.1]{TO} we denote by 
\begin{equation}\label{def_Y_to_la}
\BY(\la):=\{(i,j)\,|\, 1 \leqslant i \leqslant l,\, 1 \leqslant j \leqslant \la_i\}
\end{equation}
the corresponding Young diagram.  If $\square=(i,j) \in \BY(\la)$ is a cell let $c(\square):=j-i$ be the {\emph{content}} of $\square$. 
By  a \emph{hook} associated to the cell $(i,j)$ we call the set 
\begin{equation*}
\BH_{(i,j)}:=\{(i,j)\} \cup \{(i',j) \in \BY(\la)\,|\,i'>i\} \cup \{(i,j') \in \BY(\la)\,|\,j'>j\}.
\end{equation*}
Box $(i,j)$ is called the \emph{root} of the hook $\BH_{(i,j)}$. By a \emph{Frobenius hook} of $\BY(\la)$ we mean a hook of the form $\BH_{(i,i)}$. The diagram $\BY(\la)$ is the disjoint union of its Frobenius hooks. Suppose that $(1,1),(2,2),\ldots,(s,s)$ are  cells of $\BY(\la)$ with zero content. Let $\BH_i$ be the Frobenius hook with root $(i,i)$. Let $k_i$ be the height of $\BH_i$ and $n_i$ be the size of $\BH_i$.  
Now we are ready to describe the tuple $[(X^\la,Y^\la)] \in \CC(n)^{\BT}$ corresponding to $\la$.
\begin{Def}\label{def_X_Y_la}
Tuple $(X^\la,Y^\la)$ is defined as follows. We have $X^\la=D_n$.
The $n_i \times n_i$ diagonal blocks of $Y^\la$ are given by the matrices $Y(n_i,k_i)$, and the off-diagonal blocks satisfy the following property: 
For $i\neq j$, $(Y^\la)_{ij}$ is the unique  $n_i \times n_j$ matrix with non-zero entries on the diagonal $k_i-k_j$, satisfying the following property:
\begin{equation*}
(Y^\la)_{ij}D_{n_j}- D_{n_i}(Y^\la)_{ij}=n_iE_{k_i\,k_j}.
\end{equation*}
\end{Def}

\begin{Rem}
{\emph{If $\la \in \CP(n)$ is a hook of   height $k$, then we have  $Y^\la=Y(n,k)$.  Note that the diagonal matrix elements of $Y^\la X^\la=Y(n,k)D_n$ are precisely the contents of cells of the hook $\la$.
}}
\end{Rem}

\begin{Prop}
The eigenvalues of $Y^\la X^\la=Y^\la D_n$ are the same as the eigenvalues of blocks $(Y^\la)_{ii}D_{n_i}=Y(n_i,k_i)D_{n_i}$ as if off-diagonal blocks in $Y^\la X^\la$ were not present. So, the eigenvalues of $Y^\la X^\la$ are  diagonal elements of $Y(n_i,k_i)D_{n_i}$ that are exactly the multiset of contents of boxes of $\la$.
\end{Prop}

\begin{proof}
Follows from the proof of \cite[Proposition 6.13]{WI}.
\end{proof}

\begin{Cor}\label{symmetric_of_JM_to_roots}
The isomorphism $\Theta\colon Z(\BC S_n) \iso \BC[\mathcal{C}(n)]^{\BT}$ sends \newline  $f(JM_1,\ldots,JM_n)$ to the restriction of the function 
\newline $[(X,Y)] \mapsto f(\al_1,\ldots,\al_n)$ to $\CC(n)^{\BT}$, where $f$ is a 
symmetric function on $n$ variables.
\end{Cor}

\subsection{Proof of Theorem~\ref{rees_Z_via_schem_fixed}}
We have already constructed (see Section~\ref{sect_idea_iso_schematic_center}) the isomorphism
\begin{equation*}
\Theta\colon Z(\BC S_n) \iso \BC[\CC(n)^{\BT}]
\end{equation*}
 and have shown that this isomorphism sends generators  \newline $f(JM_1,\ldots,JM_n) \in Z(\BC S_n)$ to  functions \newline $\CC(n)^{\BT} \ni (X,Y) \mapsto f(\al_1,\ldots,\al_n)$ (see Corollary~\ref{symmetric_of_JM_to_roots}). To finish the proof of Theorem~\ref{rees_Z_via_schem_fixed} it remains to show that the isomorphism $\Theta$ is filtered. Let us  note that if $f$ has degree $k$ then $\on{deg}f(JM_1,\ldots,JM_n)=2k$ and the degree of the function $(X,Y) \mapsto f(\al_1,\ldots,\al_n)$ is also equal to $2k$. Thus, in order to show that $\Theta$ is filtration-preserving it is enough to check that 
 $F_{2k}Z(\BC S_n)$ is generated (as a vector space) by 
 \begin{multline*}
\Big\{f(JM_1,\ldots,JM_n)\,| \\
| \, f~\text{is a homogeneous symmetric polynomial of degree}\,\leqslant k\Big\}.
\end{multline*}
 This is a direct corollary of the following (classical) proposition.
\begin{Prop}
The algebra $\on{gr}Z(\BC S_n)$ is generated by the elements 
\begin{equation*}
\Big\{\on{gr}\Big(f(JM_1,\ldots,JM_n)\Big)\,|\, f~\text{is  a homogeneous symmetric polynomial}\Big\}.
\end{equation*}
\end{Prop}
\begin{proof}
The claim follows from the proof of~\cite[Theorem~1.9]{MU}.
\end{proof}

We finish this section with the following conjecture.
\begin{Conj}
The isomorphism $$\on{gr}\Theta\colon \on{gr}Z(\BC S_n) \iso \BC[(S^n(\BA^2))^{\BT}]$$ coincides with the isomorphism constructed in~\cite[Section 2]{HI}. 
\end{Conj}

\section{Description of the symplectic dual to the Gieseker variety}\label{two_descriptions)of_sympl_dual_section}
Recall that the Gieseker variety is the Nakajima quiver variety corresponding to the Jordan quiver (see Definition~\ref{def_gieseker}). It depends on the pair $n,r \in \BZ_{\geqslant 1}$ and is denoted by $\mathfrak{M}(n,r)$. The corresponding affine Poisson variety is denoted by $\mathfrak{M}_0(n,r)$.
In this section, we give two different ways to describe the symplectically dual variety $\mathfrak{M}_0(n,r)^{!}$ and its deformation $\mathfrak{M}_0(n,r)^{!}_{\mathfrak{a}}$.

In the paper~\cite{BFN}, the candidate for symplectically dual variety to every Nakajima quiver variety was constructed. 

\subsection{Construction of Coulomb branch}\label{dual_via_Coulomb_constr_tripl_hom} Let us recall the construction in our case (when we start from the Jordan quiver with the dimension vector $n \in \BZ_{\geqslant 1}$ and framing $r \in \BZ_{\geqslant 1}$). 

Recall  the vector space ${\bf{N}}=\on{Hom}(V,V) \oplus \on{Hom}(V,W)$ and the group $G_{n}=\on{GL}(V)$ acting on ${\bf{N}}$ (see Section \ref{naka_quiver_def}).

\begin{Def}        We define $\on{Gr}_{\on{GL}(V)}$ as the moduli space of the data $(\CP,\varphi)$, where:
        \item
    (a) $\mathcal{P}$ is a $\on{GL}(V)$-bundle  on $\mathbb{P}^{1};$
    \item
    (b) $\varphi\colon \mathcal{P}^{\mathrm{triv}}|_{\mathbb{P}^{1} \setminus \{0\}} \iso \mathcal{P}|_{\mathbb{P}^{1} \setminus \{0\}}$ is 
    a trivialization of $\CP$ restricted to $\BP^1 \setminus \{0\}$.
\end{Def}

We then consider the moduli space of triples $\CR_{n,r}$ (corresponding to the Jordan quiver, dimension vector $n$ and framing $r$) defined as follows.
\begin{Def}
Let $\CR_{n,r}$ be the moduli space of triples 
$\{(\CP,\varphi,s)\,\}$, where $(\CP,\varphi)$ is a point of $\on{Gr}_{\on{GL}(V)}$ and $s$ is a section of the associated vector bundle $\CP_{{\bf{N}}}=\CP \times_{\on{GL}(V)} {\bf{N}}$ such that it is sent to a regular section of a trivial bundle under $\varphi$.  
\end{Def}

Set $\on{GL}(V)_{\CO}:=\on{GL}(V)[[z]]$. We can consider the equivariant Borel-Moore homology $H^{\on{GL}(V)_{\CO}}_*(\CR_{n,r})$ (see~\cite[Section 2(ii)]{BFN} for the definition and detailed discussion). This vector space is equipped with an algebra structure via convolution $*$ (see~\cite[Section 3]{BFN}). It follows from ~\cite[Proposition 5.15]{BFN} that the algebra $(H^{\on{GL}(V)_{\CO}}_*(\CR_{n,r}),*)$ is commutative.
\begin{Def}
The Coulomb branch $\CM(n,r)$ is defined as the spectrum of the  algebra $H^{{\on{GL}(V)}_\CO}_*(\CR_{n,r})$:
\begin{equation*}
\CM(n,r):=\on{Spec}(H_*^{\on{GL}(V)_\CO}(\CR_{n,r})).   
\end{equation*}
\end{Def}

The variety $\CM(n,r)$ is conjectured to be symplectically dual to $\mathfrak{M}(n,r)$. The 
 deformation $\CM(n,r)_{\mathfrak{a}} \ra \mathfrak{a}$ of $\CM(n,r)$ can be constructed as follows (compare with \cite[Section 3(viii)]{BFN}). Let $\BT$ be the copy of $\BC^\times$. Let $T_r \subset \on{SL}(W)$ be a maximal torus. We have the natural action of $A=\BT \times T_r $ on ${\bf{N}}$ given by 
\begin{equation*}
(t,g) \cdot (X,\gamma)=(tX,\gamma \circ g^{-1}),\, (t,g) \in \BT \times T_r.
\end{equation*}

\begin{Warning}
Note that the action of ${\mathbb{T}}$ on ${\bf{N}}$ is not the scaling action (as in \cite{BFN}). This action, for example, appears in \cite[Section 5.5]{BEF}.
\end{Warning}

We can identify $\on{Lie}\BT \simeq \BA^1$, so $\BC[\on{Lie}\BT]=\BC[\kappa]$ for some variable $\kappa$.  We can also identify $\mathfrak{t}_r:=\on{Lie}T_r$ with the subspace of $\BC^r$ consisting of points with the sum of coordinates being zero, i.e., $\BC[\mathfrak{t}_w]=\BC[a_1,\ldots,a_r]/(a_1+\ldots+a_r)$.

\begin{Def}
We define
\begin{equation*}
\CM(n,r)_{\mathfrak{a}}:=\on{Spec}(H_*^{A \times \on{GL}(V)_\CO}(\CR_{n,r})).    
\end{equation*}
\end{Def}
Let us now give a more explicit description of the algebra \newline $\BC[\CM(n,r)_{\mathfrak{a}}]$. We recall its realization as a spherical subalgebra of the rational Cherednik algebra $H_{n,r}$ corresponding to the group $\Gamma_n=S_n \ltimes (\BZ/r\BZ)^n$.

\subsection{Double affine Rational Cherednik algebra corresponding to $S_n \ltimes (\BZ/r\BZ)^n$}\label{DAHA_corresp_to_Gamma_section}
We start by recalling some definitions and notations. 

Consider the subgroup $\Gamma_n \subset \on{GL}_n$ of monomial matrices with entries being $r$th roots of unity. Let $\eta \in \BC^\times$ be a $r$th primitive root of unity. We set 
$
\epsilon_j:=\on{diag}(1,\ldots,1,\eta,1,\ldots,1).    
$
Note that we have the natural embedding $S_n \subset \Gamma_n$. We obtain the  identification $S_n \ltimes (\BZ/r\BZ)^n \iso \Gamma_n$.
Consider the standard representation $\Gamma_n \curvearrowright \mathfrak{h}:=\BC^n$ induced by the embedding $\Gamma_n \subset \on{GL}_n$.
Let $x_1,\ldots,x_n$ be the standard basis in $\mathfrak{h}=\BC^n$ and denote by $y_1,\ldots,y_n \in \mathfrak{h}^*$ the dual basis.

Let $t,\kappa,c_1,\ldots,c_{r-1}$ be formal parameters. 
We set 
\begin{equation*}
{\bf{h}}:=
\BC[\kappa,c_1,\ldots,c_{r-1}], 
\ \ c(q):=\sum_{i=1}^{r-1}c_iq^i,
\end{equation*}
here $q$ is a formal variable.
Following~\cite{webster} and~\cite[Section 2.3]{TO} we define the rational 
Cherednik algebra $\mathcal{H}_{\Gamma_n}$ in the following way.

\begin{Def}
Algebra $\CH_{\Gamma_n}=\CH_{n,r}$ is  a quotient of the semi-direct product
\begin{equation*}
\Big(\BC \Gamma \ltimes T^{\bullet}(\mathfrak{h} \oplus \mathfrak{h}^*) \Big) \otimes {\bf{h}}[\hbar]
\end{equation*}
subject to the relations:
\begin{equation}\label{rel_xx_yy}
[x_i,x_j]=[y_i,y_j]=0,
\end{equation} 
\begin{equation}\label{comm_x_y_H}
[x_i,y_i]=-\hbar-\kappa \sum_{j \neq i}\sum_{p=0}^{r-1}(ij)\epsilon_i^p\epsilon_j^{-p}-c(\epsilon_i),    \end{equation}
\begin{equation}\label{comm_x_y_dif_H}
[x_i,y_j]=\kappa \sum_{p=0}^{r-1}\eta^{p}(ij) \epsilon_i^p \epsilon_j^{-p}  \quad \quad (i \neq j).
\end{equation}
We also set $H_{n,r}:=\CH_{n,r}/(\hbar)$.
\end{Def}

\begin{Rem}
{\emph{
Our parameters $(\hbar,\kappa,c_1,\ldots,c_{r-1})$ and the  parameters \newline   $(\ul{t},\ul{\kappa},\ul{c}_1,\ldots,\ul{c}_{r-1})$ of~\cite{MM} are related  as follows: $\ul{t}=-\hbar$, $\ul{\kappa}=\kappa$, $\ul{c}_i=\frac{c_i}{(1-\eta^{-i})}$ (note that $x_i^{Martino}=y_i$, $y_i^{Martino}=x_i$). 
}}
\end{Rem}

\begin{Def}
Set ${\bf{e}}:=\frac{1}{|\Gamma_n|}\sum_{g \in \Gamma_n}g$. We denote by $\CH^{\mathrm{sph}}_{n,r}$ the spherical subalgebra ${\bf{e}}\CH_{n,r} {\bf{e}} \subset \CH_{n,r}$, and by $H_{n,r}^{\mathrm{sph}}$ the subalgebra ${\bf{e}} H_{n,r} {\bf{e}} \subset H_{n,r}$.
\end{Def}

Recall that $Z(H_{n,r}) \subset H_{n,r}$ is the center.
The following proposition holds by~\cite{EG}.
\begin{Prop}\label{center_via_spherical}
The composition $Z(H_{n,r}) \subset H_{n,r} \twoheadrightarrow {\bf{e}}H_{n,r}{\bf{e}}$ induces the identification $Z(H_{n,r}) \iso {\bf{e}}H_{n,r}{\bf{e}}$.
\end{Prop}

The algebra $H_{n,r}$ is $\BZ$-graded as follows:
\begin{equation}\label{Z_grading_on _H}
\on{deg}x_j=1,\, \on{deg}y_j=-1,\, \on{deg}\Gamma_n=\on{deg}{\bf{h}}=0.
\end{equation}

\subsection{Coulomb branch for Jordan quiver as the center of rational Cherednik algebra for $S_n \ltimes (\BZ/r\BZ)^n$}

Following \cite{webster} set
\begin{equation*}
p(q):=\frac{1}{r}\sum_{l=1}^{r-1} \frac{c_l}{\eta^{-l}-1}q^l=\frac{1}{r}\sum_{l=1}^{r-1}\frac{c_l\eta^{l}}{1-\eta^{l}}q^l.
\end{equation*}
Then we can write 
\begin{equation}\label{c_vs_p}
c(q)=r(p(\eta^{-1} q)-p(q)).
\end{equation}

\begin{Prop}\label{iden_Coulomb_Cherednik!!}
There exists the isomorphism of algebras 
\begin{equation*}
H_{n,r}^{\mathrm{sph}} 
\simeq H_*^{A \times \on{GL}(V)_{\CO}}(\CR_{n,r})    
\end{equation*}
that identifies parameters in the following way $\kappa=\kappa$, $a_i=p(\eta^{i-1})$. The isomorphism above sends $e_i(u_1,\ldots,u_n)$ to $m_i$, where $m_i=c_i(V) * 1$.
\end{Prop}

The Proposition \ref{iden_Coulomb_Cherednik!!} was proven in \cite[Theorem 1.1]{KONA}, see also \cite[Theorem 4.1]{webster} and \cite[Theorem 2.19]{BEF}.

\begin{Example}
Let us illustrate how (quantized version of) the isomorphism $H_{n,r}^{\mathrm{sph}} 
\simeq H_*^{A \times \on{GL}(V)_{\CO}}(\CR_{n,r})$ works for $n=1$. We have $A=\BC^\times \times (\BC^\times)^{r-1}$ and ${\mathbf{N}}=\BC \oplus (\BC^{r})^*$.  By \cite[Section 4(iii)]{BFN} the algebra $H_*^{A \times (\on{GL}(V)_{\CO} \rtimes \BC^\times)}(\CR_{n,r})$ has the following description: it is generated over $\BC[\kappa,a_1,\ldots,a_r]/(a_1+\ldots+a_r)$ by $r_1, r_{-1}, b$ subject to relations:
\begin{equation*}
r_1r_{-1}=\prod_{i=1}^r(b-a_i),\, r_{-1}r_1=\prod_{i=1}^r(b-a_i-\hbar),\,[r_1,b]=\hbar r_1,\, [r_{-1},b]=-\hbar r_{-1}.    
\end{equation*}

The algebra $H_{1,r}$ is generated by $\BC \BZ/r\BZ$ and $x,y$ subject to the relations 
\begin{equation*}
[x,y]=-\hbar-c(\epsilon),\, \epsilon x=\eta x \epsilon,\, \epsilon y = \eta^{-1} y \epsilon.    
\end{equation*}
Set $u:=\frac{1}{r}(xy+\hbar)+p(\eta^{-1}\epsilon)$ (compare with Definition \ref{OD_elements_DEF!}). 
The isomorphism 
\begin{equation*}
H_*^{A \times (\on{GL}(V)_{\CO} \rtimes \BC^\times_\hbar)}(\CR_{1,r}) \iso \mathcal{H}^{\mathrm{sph}}_{1,r}
\end{equation*}    
is then given by 
\begin{equation*}
r_{-1} \mapsto {\bf{e}}x^r{\bf{e}},\,  r_{1} \mapsto \frac{1}{r^r}{\bf{e}}y^r{\bf{e}},\, b \mapsto {\bf{e}}u{\bf{e}},\, \kappa \mapsto \kappa,\, a_i \mapsto p(\eta^{i-1})-\frac{(i-1)\hbar}{r}.   
\end{equation*}
\end{Example}

We conclude (using Proposition \ref{center_via_spherical} and Proposition \ref{iden_Coulomb_Cherednik!!})  that the deformation $\CM(n,r)_{\mathfrak{a}}$  of the variety $\CM(n,r)$ can be described as
\begin{equation*}
\CM(n,r)_{\mathfrak{a}}=\on{Spec}(Z(H_{n,r})).    
\end{equation*}

\section{Equivariant cohomology $H^*_{A}(\mathfrak{M}(n,r),\BC)$}\label{equivariant_cohomology_embedding}

We start by recalling the  combinatorial parametrization of the $A=\BT \times T_r$-fixed points of the Gieseker variety $\mathfrak{M}(n,r)$. 
	
Consider the case $r=1$. The variety $\mathfrak{M}(n,1)$ coincides with the Hilbert scheme $\on{Hilb}_{n}(\mathbb{A}^{2})$  (see Proposition \ref{gieseker_via_hilbert_prop} above or \cite[Section 2.2]{NA0}). Recall that this Hilbert scheme parametrizes the codimension $n$ ideals in $\C[x,y]$. 
The torus $T_1$ is zero-dimensional, thus $$(\on{Hilb}_{n}(\BA^{2}))^{\BT \times T_1}=(\on{Hilb}_{n}(\BA^{2}))^{\BT}.$$ 
	The fixed point set $(\on{Hilb}_{n}(\BA^{2}))^{\BT}$ is the set of monomial ideals $J \in \on{Hilb}_{n}(\BA^{2})$.  Such ideals are parametrized by the set $\CP(n)$ of partitions of $n$ as follows.   Recall that (following notations of \cite[Section 4.1]{TO}) we associate to $\la=(\la_1,\ldots,\la_l) \in \CP(n)$ the Young diagram
	\begin{equation*}
	\BY(\la)=\{(i,j)\, | \, 1 \leqslant i \leqslant l,\, 1 \leqslant j  \leqslant \la_i\}.
	\end{equation*}
	
	We fill $\BY(\la)$ with monomials by putting $x^{i-1}y^{j-1}$ into the box $(i,j) \in \BY(\la)$.
	The ideal $J_{\la}$ that corresponds to $\la$ is spanned by the monomials outside the diagram.
	We have the following classical result.
	
	\begin{Prop}\label{Fixed Hilbert}
	The fixed point set $(\on{Hilb}_n(\BA^2))^{\BT}$ is identified with the set $\CP(n)$ via the map $\la \mapsto J_{\la}$ described above.
	\end{Prop}

	Now we proceed to the case of arbitrary $r$. Let us introduce some notation.
	Denote by $w_1,\ldots,w_r \in W$ a basis of $W$ consisting of eigenvectors of $T_r$.

	The following lemma is classical (see, for example, \cite{NAYO}).     
	\begin{Lemma} \label{1} The variety $\mathfrak{M}(n,r)^{T_r}$ is $\set{T}$-equivariantly isomorphic to the disjoint union 
	\begin{equation*}
	\bigsqcup\limits_{\sum\limits_{l=1}^{r} n_{l}=n}~\prod\limits_{l=1}^{r} \mathfrak{M}({n_{l}},1)~ \simeq 
	\bigsqcup\limits_{\sum\limits_{l = 1}^{r} n_{l}=n}~\prod\limits_{l =1}^{r} \on{Hilb}_{n_l}(\BA^2).
	\end{equation*}
	\end{Lemma}

	\begin{Def}\label{multipart}
	We say that an ordered $r$-tuple ${\boldsymbol{\la}}=(\la^0,\ldots,\la^{r-1})$ of partitions
	defines an $r$-multipartition of  $n\in \mathbb{Z}_{\geqslant 0}$ if $\sum\limits_{l=0}^{r-1}|\la^l|=n$.  Let $\CP(r,n)$ denote the set of $r$-multipartitions of $n$.
	\end{Def}
	
	\begin{Prop} \label{FP on Gi}
	The set of fixed points $\mathfrak{M}(n,r)^{A}$
	is identified with the set $\mathcal{P}(r,n)$. A multipartition ${\boldsymbol{\la}}=(\la^0,\la^1,\ldots,\la^{r-1})$, corresponds to the following quiver data
	$V,W,\,
	X^{\boldsymbol{\lambda}},Y^{\boldsymbol{\lambda}} \in \on{End}(V),\,
	\gamma^{\boldsymbol{\lambda}} \in \on{Hom}(W,V),\, 
 \newline \delta^{\boldsymbol{\lambda}} \in \on{Hom}(V,W)$:
	\begin{equation*} 
	V:=\bigoplus_{l=0}^{r-1} \BC[x,y]/J_{\la^l},~
	W= \bigoplus_{l=0}^{r-1} \BC w_l;
	\end{equation*}
	\begin{equation*}
	X^{\boldsymbol{\la}} := \bigoplus_{i=0}^{r-1} L^{l}_{x},~
	Y^{\boldsymbol{\la}} := \bigoplus_{l=0}^{r-1} L^{l}_{y};
	\end{equation*}
	\begin{equation*}
	\gamma^{\boldsymbol{\la}} \hspace{0,1cm} \on{sends} \hspace{0,1cm} w_{i} \in W \hspace{0,1cm} \on{to} \hspace{0,1cm} [1] \hspace{0,1cm} \on{in} \hspace{0,1cm} \BC[x,y]/J_{\la^i},~
	\delta^{\boldsymbol{\la}} = 0,
	\end{equation*}
by $L^{l}_{x},\,L^{l}_{y}$ we denote operators of  multiplications by $x,y$ on $\BC[x,y]/J_{\la^l}$.

	\end{Prop}
	
	\begin{proof}
	This follows from Lemma~\ref{1} and the description of $\BT$-fixed points of Hilbert schemes above (see Proposition \ref{Fixed Hilbert}).
	\end{proof}

We have the $A$-equivariant embedding 
\begin{equation*}
\iota\colon \mathfrak{M}(n,r)^{A}   \subset \mathfrak{M}(n,r).
\end{equation*}
This embedding induces the homomorphism 
\begin{equation*}
\iota^*\colon H^*_{A}(\mathfrak{M}(n,r)) \ra  H^*_{A}(\mathfrak{M}(n,r)^{A})=\BC[\mathfrak{a}]^{\oplus |\CP(r,n)|}=E.   
\end{equation*}

\begin{Lemma}
The homomorphism $\iota^*$ is an embedding that becomes isomorphism after tensoring by $F=\BC(\mathfrak{a})$.
\end{Lemma}
\begin{proof}
Follows from the Atiyah-Bott localization theorem (see \cite{AB}) together with the fact that $H^*_{A}(\mathfrak{M}(n,r),\BC)$ is a free $\BC[\mathfrak{a}]$-module (see~\cite[Theorem 7.3.5]{NA_quant_affine}). 

\end{proof}

\begin{Rem}\label{more_car_loc_statement}
{\emph{Actually, we do not need to tensor by the whole $\BC(\mathfrak{a})$ but only have to localize at elements $h \in \BC[\mathfrak{a}]$ corresponding to the cocharacters $\nu\colon \BC^\times \ra A$ such that $\mathfrak{M}(n,r)^{\nu(\BC^\times)}$ is {\em{infinite}}. 
These elements $h$ can be described explicitly.}}
\end{Rem}

Recall now that by
\cite{MN} the algebra $H^*_{A}(\mathfrak{M}(n,r),\BC)$ is generated by the $A$-equivariant Chern classes $c_k(\mathcal{V})$, $k=1,\ldots,n$, where  $\mathcal{V}$ is the tautological rank $n$ vector bundle on $\mathfrak{M}(n,r)$.

\begin{Lemma}
The image $\iota^*(c_k(\mathcal{V}))$ is equal to the collection 
\begin{equation*}
(e_{k}(\kappa c^0_1+a_1,\ldots,\kappa c^{0}_{|\la^0|}+ a_1,\ldots,\kappa c^{r-1}_1+ a_r,\ldots,\kappa c^{r-1}_{|\la^{r-1}|}+a_r))_{\boldsymbol{\la} \in \CP(r,n)},
\end{equation*}
where for $l=0,1,\ldots,r-1$ $c^l_1,c^l_2,\ldots,c^l_{|\la^l|}$ is the multiset of contents of boxes of the diagram $\BY(\la^l)$.
\end{Lemma}
\begin{proof}
The homomorphism $\iota^*$ sends $c_k(\CV)$ to 
\begin{equation*}
c_k(\CV|_{\mathfrak{M}(n,r)^{A}})=(c_k(\CV|_{(X^{\boldsymbol{\la}},Y^{\boldsymbol{\la}},\gamma^{\boldsymbol{\la}},0)}))_{\boldsymbol{\la} \in \CP(r,n)}.
\end{equation*}
Note that $c_k(\CV|_{(X^{\boldsymbol{\la}},Y^{\boldsymbol{\la}},\gamma^{\boldsymbol{\la}},0)})$ is nothing but $e_k(\al_1,\ldots,\al_n)$, where \newline $\al_1,\ldots,\al_n$ is the multiset of $\mathfrak{a}$-weights of $\CV|_{(X^{\boldsymbol{\la}},Y^{\boldsymbol{\la}},\gamma^{\boldsymbol{\la}},0)}$. 

Recall that $\BC[\on{Lie}\BT]=\BC[\kappa]$, $\BC[\mathfrak{t}_{r}]=\BC[a_1,\ldots,a_{r}]/(a_1+\ldots+a_r)$ and $\mathfrak{a}=(\on{Lie}\BT) \oplus \mathfrak{t}_r$.
We claim that the $\mathfrak{a}$-weight of $[x^iy^j] \in \BC[x,y]/J_{\la^l}$ is equal to $\kappa(j-i)+a_l$ and this will conclude the proof.
Indeed, taking $g=(t^{\kappa},t^{a_1},\ldots,t^{a_r}) \in \BT \times T_r$ we see that this element acts on $(X^{\boldsymbol{\la}},Y^{\boldsymbol{\la}},\gamma^{\boldsymbol{\la}},0)$ in the following way:
\begin{equation*}
g \cdot X^{\boldsymbol{\la}}=\bigoplus_{l=0}^{r-1}t^{\kappa}L^l_x,\, g \cdot Y^{\boldsymbol{\la}}=\bigoplus_{l=0}^{r-1} t^{-\kappa}    L^l_y,\,(g \cdot \gamma^{\boldsymbol{\la}})(w_l)=t^{-a_l}w_l.
\end{equation*}
Consider the element $g' \in \prod_{l=0}^{r-1}\on{GL}(\BC[x,y]/J_{\la^l})$ that acts on $x^iy^j \in \BC[x,y]/J_{\la^l}$ via the multiplication by $t^{a_l+\kappa j-\kappa i}$.
Directly from the definitions, we see that 
\begin{equation*}
g \cdot (X^{\boldsymbol{\la}},Y^{\boldsymbol{\la}},\gamma^{\boldsymbol{\la}},0) = g' \cdot (X^{\boldsymbol{\la}},Y^{\boldsymbol{\la}},\gamma^{\boldsymbol{\la}},0).
\end{equation*}
The claim follows since the element $g \in {\mathbb{T}} \times T_r$ acts on the fiber $V$ of $\mathcal{V}$ at $(X^{\boldsymbol{\la}},Y^{\boldsymbol{\la}},\gamma^{\boldsymbol{\la}},0)$ via $g'$.
\end{proof}

Combining the results of this section we obtain the following proposition.
\begin{Prop}\label{emb_cohom_prop_use}
The homomorphism $$\iota^*\colon H^*_A(\mathfrak{M}(n,r)) \hookrightarrow H^*_A(\mathfrak{M}(n,r)^A)=E$$ is injective and becomes an isomorphism after tensoring by $F=\BC(\mathfrak{a})$. On generators $c_k(\CV),\, k=1,\ldots,n$ it is given by: 
\begin{multline*}
c_k(\CV) \mapsto \\ 
\mapsto (e_{k}(\kappa c^0_1+a_1,\ldots,\kappa c^{0}_{|\la^0|}+ a_1,\ldots,\kappa c^{r-1}_1+ a_r,\ldots,\kappa c^{r-1}_{|\la^{r-1}|}+a_r))_{\boldsymbol{\la} \in \CP(r,n)}.   
\end{multline*}
\end{Prop}

\section{The center $Z(R^r(n))^{JM}$ of the cyclotomic degenerate Hecke algebra}\label{rep_theory_of_cyclotomic}

Recall the space ${\bf{k}}=\BC[\kappa,a_1,\ldots,a_r]/(a_1+\ldots+a_r)$.
\begin{Def}\label{cyclot_degen_defe}
The degenerate affine Hecke algebra $R(n)=R(S_n)$ is generated by $\BC S_n$ and ${\bf{k}}[z_1,\ldots,z_n]$ subject to relations: 
\begin{equation*}
s_iz_j=z_{s_i(j)}s_i+\kappa(\delta_{i+1,j}-\delta_{i,j}).
\end{equation*}
The cyclotomic degenerate Hecke algebra $R^r(n)$ is the quotient of $R(n)$ by the ideal generated by $\prod_{i=1}^r (z_1-a_i)$. 

\end{Def}

\begin{Def}
Let $Z(R^r(n))^{JM} \subset R^r(n)$ be the image of \newline ${\bf{k}}[z_1,\ldots,z_n]^{S_n} \subset R(n)$ in $R^r(n)$.
\end{Def}

\begin{Rem}
{\emph{By \cite[Theorem 6.5]{L0}, the subalgebra ${\bf{k}}[z_1,\ldots,z_n]^{S_n} \subset R(n)$ is nothing else but the center of $R(n)$. It follows from \cite[Theorem 1]{Br} that the algebra $Z(R^r(n))^{JM}$ is  the center of $R^r(n)$.}}
\end{Rem}

To every ${\boldsymbol{\la}} \in \CP(r,n)$ one can associate the ``universal'' Specht module $\widetilde{S}_{\bf{k}}({\boldsymbol{\la}})$ over $R^r(n)$ (compare with~\cite[Section~4]{Br}) in the following way.

Recall that ${\boldsymbol{\la}}=(\la^0,\la^1,\ldots,\la^{r-1})$. For $i=0,1,\ldots,r-1$  set $n_i:=|\la^i|$. We  slightly modify the algebras $R(n)$, $R^r(n)$ first.

\begin{Def} Let $R_{\BC[\kappa,a_1,\ldots,a_r]}(n)$ be the algebra generated by $\BC S_n$ and \newline  $\BC[z_1,\ldots,z_n]$ over $\BC[\kappa,a_1,\ldots,a_r]$ subject to the relations 
\begin{equation*}
s_i z_j=z_{s_i(j)}s_i+\kappa (\delta_{i+1,j}-\delta_{i,j}).    
\end{equation*}
We denote by ${R}^r_{\BC[\kappa,a_1,\ldots,a_r]}(n)$ the quotient of $R_{\BC[\kappa,a_1,\ldots,a_r]}(n)$ by the ideal generated by $\prod_{i=1}^r (z_1-a_i)$.
\end{Def}

For every $i=0,1,\ldots,r-1$ consider the usual Specht module $S_{n_i} \curvearrowright S(\la^i)$, corresponding to $\la^i$. We can extend $S(\la^i) \otimes \BC[\kappa,a_i]$ to the $R_{\BC[\kappa,a_i]}(n_i)$-module by letting $z_1$ act via the multiplication by  $a_i$. Note that we have the natural embedding 
\begin{multline*}
R_{\BC[\kappa,a_1,\ldots,a_r]}(n_0,\ldots,n_{r-1}):= \\ =R_{\BC[\kappa,a_1]}(n_0) \otimes_{\BC[\kappa]} \ldots \otimes_{\BC[\kappa]} R_{\BC[\kappa,a_{r}]}(n_{r-1}) \subset R_{\BC[\kappa,a_1,\ldots,a_r]}(n),
\end{multline*}
induced by the embedding $S_{n_0} \times \ldots \times S_{n_{r-1}} \subset S_n$. Then we define 
\begin{multline*}
\widetilde{S}_{\BC[\kappa,a_1,\ldots,a_r]}(\boldsymbol{\la}):=
R_{\BC[\kappa,a_1,\ldots,a_r]}(n) \otimes_{R_{\BC[\kappa,a_1,\ldots,a_r]}(n_0,\ldots,n_{r-1})} \\ \otimes \Big((S({\la^0}) \otimes \BC[\kappa,a_1]) \otimes_{\BC[\kappa]} \ldots \otimes_{\BC[\kappa]} (S(\la^{r-1}) \otimes \BC[\kappa,a_r])\Big)
\end{multline*}
that is an ${R}^r_{\BC[\kappa,a_1,\ldots,a_r]}(n)$-module. Modding out by $(a_1+\ldots+a_r)$ 
\newline we obtain the desired $R^r(n)$-module $\widetilde{S}_{\bf{k}}({\boldsymbol{\la}})$.

Let us now compute the action of $Z(R^r(n))^{JM}$ on $\widetilde{S}_{\bf{k}}(\boldsymbol{\la})$. We start with the following lemma.
Let $\mu$ be a  partition of some $m \in \BZ_{\geqslant 1}$ and consider the action $R_{\BC[\kappa,a]}(m) \curvearrowright S(\mu) \otimes \BC[\kappa,a]$. Recall that $S(\mu)$ has a basis labeled by the standard Young tableaux. such that $JM_i$ act on such vectors with the content of $i's$ entry.
\begin{Lemma}\label{act_z_i_via_cyclotomic}
 Let $B$ be a standard Young tableau on $\mu$ and recall that $p_B \in S(\mu)$ is the corresponding vector. The element $z_i \in R_{\BC[\kappa,a]}(m)$ acts on $p_B$ via the multiplication by $\kappa \on{ct}(B(i))+a$.
\end{Lemma}
\begin{proof}
The $i$-th Jucys-Murphy element is given by 
\begin{equation*}
JM_i=\sum_{j<i}(ij) \in \BC S_m.    
\end{equation*}
Recall now that 
$
s_iz_{i+1}=z_is_i+\kappa
$   so we have 
\begin{equation*}
z_{i+1}=s_iz_is_i+s_i\kappa.    
\end{equation*}
It follows that 
\begin{equation*}
z_j=s_{j-1}s_{j-2}\ldots s_1 z_1 s_1\ldots s_{j-1} +\kappa JM_i    
\end{equation*}
so the action of $z_j$ on $S(\mu) \otimes \BC[\kappa,a]$ coincides with the action of $\kappa JM_i+a$ and this concludes the proof.
\end{proof}

From Lemma~\ref{act_z_i_via_cyclotomic} we obtain the following proposition  (see also~\cite[Section 4]{Br}).
\begin{Prop}\label{central_act_via_cyclotom_rep}
The class $[f(z_1,\ldots,z_n)] \in Z(R^r(n))^{JM}$ of the element $f$ in $\BC[z_1,\ldots,z_n]^{S_n} \subset R(n)$ acts on the representation $\widetilde{S}_{\bf{k}}({\boldsymbol{\la}})$ via multiplication by 
\begin{equation*}
f(\kappa c^0_1+a_1,\ldots,\kappa c^{0}_{|\la^0|}+ a_1,\ldots,\kappa c^{r-1}_1+ a_r,\ldots,\kappa c^{r-1}_{|\la^{r-1}|}+a_r),
\end{equation*}
where $c^l_1,c^l_2,\ldots,c^l_{|\la^l|}$ is the multiset of contents of boxes of the diagram $\BY(\la^l)$.
\end{Prop}

Since every element of $Z(R^r(n))^{JM}$ is central we can consider the homomorphism ${\psi}\colon Z(R^r(n))^{JM} \ra \bigoplus_{{\boldsymbol{\la}}} \on{End}_{R^r(n)}(\widetilde{S}_{\bf{k}}(\boldsymbol{\la}))$ induced by the action of $Z(R^r(n))^{JM}$ on representations $\widetilde{S}_{\bf{k}}(\boldsymbol{\la})$.

\begin{Lemma}
We have $\on{End}_{R^r(n)}\widetilde{S}_{\bf{k}}(\boldsymbol{\la})={\bf{k}}$.
\end{Lemma}
\begin{proof}
Let us first of all note that $\widetilde{S}_{\bf{k}}(\boldsymbol{\la})$ is a free ${\bf{k}}$-module (see \cite{BK} and references therein). It follows that we have an embedding 
\begin{equation*}
\on{End}_{R^r(n)}(\widetilde{S}_{\bf{k}}(\boldsymbol{\la})) \subset \on{End}_{R^r(n)\otimes_{{\bf{k}}} \ol{\on{Frac}({\bf{k}})}}(\widetilde{S}_{\bf{k}}(\boldsymbol{\la}) \otimes_{{\bf{k}}} \ol{\on{Frac}({\bf{k}})})=\ol{\on{Frac}({\bf{k}})},
\end{equation*}
where $\ol{\on{Frac}({\bf{k}})}$ is the algebraic closure of the field of fractions $\on{Frac}({\bf{k}})$ and the last equality holds since $\widetilde{S}_{\bf{k}}(\boldsymbol{\la}) \otimes_{{\bf{k}}} \ol{\on{Frac}({\bf{k}})}$ is simple over $R^r(n) \otimes_{{\bf{k}}} \ol{\on{Frac}({\bf{k}})}$ (see \cite[Section 4]{Br}). Now it follows that $\on{End}_{R^r(n)}(\widetilde{S}_{\bf{k}}(\boldsymbol{\la}))={\bf{k}}$.
\end{proof}

\begin{Prop}\label{psi_tilda_inj!}
The homomorphism 
\begin{equation*}
{\psi}\colon Z(R^r(n))^{JM} \ra \bigoplus_{{\boldsymbol{\la}}} \on{End}_{R^r(n)}\widetilde{S}_{\bf{k}}(\boldsymbol{\la})=E
\end{equation*}
becomes an isomorphism after tensoring by $F=\on{Frac}({\bf{k}})$. This homomorphism is injective. It sends generators $[e_k(z_1,\ldots,z_n)]$ to the collection 
\begin{equation*}
(e_{k}(\kappa c^0_1+a_1,\ldots,\kappa c^{0}_{|\la^0|}+ a_1,\ldots,\kappa c^{r-1}_1+ a_r,\ldots,\kappa c^{r-1}_{|\la^{r-1}|}+a_r))_{\boldsymbol{\la} \in \CP(r,n)},
\end{equation*}
where $c^l_1,c^l_2,\ldots,c^l_{|\la^l|}$ is the multiset of contents of boxes of the diagram $\BY(\la^l)$.
\end{Prop}
\begin{proof}
The last part of the claim is Proposition~\ref{central_act_via_cyclotom_rep}.
Recall now that by~\cite[Therorem 1]{Br} $Z(R^r(n))^{JM}$ is a free ${\bf{k}}$-module of rank $|\CP(n,r)|$. 
So injectivity of $\psi$ would follow if we show that $\psi$ becomes isomorphism after tensoring by $F$. In order to do that it is enough to check that $\psi$ becomes surjective after tensoring by $F$ (using the equality of ranks of $Z(R^r(n))^{JM}$ and $ \bigoplus_{\boldsymbol{\la}} \on{End}_{R^r(n)}(\widetilde{S}_{\bf{k}}(\boldsymbol{\la}))$ over $\bf{k}$). Surjectivity is a corollary of Proposition~\ref{central_act_via_cyclotom_rep} and Proposition~\ref{emb_cohom_prop_use}.
\end{proof}

\section{Schematic fixed points of $\on{Spec}Z(H_{n,r})$} \label{schem_fixed_Q_section}

\subsection{Standard representations of $H_{n,r}$, grading on them}\label{standard_simple_RCA}

We set $\zeta_{i,j}:=\frac{1}{r}\sum_{p=0}^{r-1}\epsilon_i^{p}\epsilon_j^{-p}$ (projector to the invariants under $\epsilon_i\epsilon_j^{-1}$). 
The Jucys-Murphy elements generalize from $S_n$ to $\Gamma_n$ as
\begin{equation}\label{JM_in_general_defin}
JM_{\Gamma_n,i}:=\sum_{j<i}\zeta_{i,j}(ij) \in \BC\Gamma_n.    
\end{equation}

Recall that $\CP(r,n)$ is the set of $r$-multipartitions of $n$ (Definition \ref{multipart}). Pick ${\boldsymbol{\la}} \in \CP(r,n)$ and consider the corresponding $r$-tuple of Young diagrams
\begin{equation}
\BY({\boldsymbol{\la}})=(\BY(\la^0),\ldots,\BY(\la^{r-1})).
\end{equation}
Given a cell $b \in \BY({\boldsymbol{\la}})$, define $\beta(b)=k$ if $b \in \BY({\boldsymbol{\la}}^{k-1})$ and $\on{ct}(b)=j-i$ if $b$ is in the $i$th row and
$j$th column of $\BY(\la^k)$. 
There is a bijection ${\boldsymbol{\la}} \mapsto S({\boldsymbol{\la}})$ from the set of $r$-partitions of $n$ to the set of irreducible $\Gamma_n$-modules such that $S(\boldsymbol{\la})$ has a basis $p_B$ indexed by standard Young tableaux $B$ on ${\boldsymbol{\la}}$, and $p_B$ is determined up to scalars by the equations (see, for example,~\cite[Equation~(2.16)]{GRI}):
\begin{equation}\label{act_JM_S_la}
JM_{\Gamma_n,i} \cdot p_B= \on{ct}(B(i))p_B,\, \epsilon_i \cdot p_B=\eta^{\beta(B(i))}p_B.
\end{equation}

We can consider $S(\boldsymbol{\la}) \otimes {\bf{h}}$ as a module over $({\bf{h}} \otimes \BC\Gamma_n) \ltimes S^\bullet \mathfrak{h}^*$ via the trivial action of the augmentation ideal of $S^\bullet \mathfrak{h}^*$. 
Let $\Delta(\boldsymbol{\la}):=\on{Ind}^{H_{n,r}}_{({\bf{h}} \otimes \BC\Gamma) \ltimes S^\bullet \mathfrak{h}^*} (S(\boldsymbol{\la}) \otimes {\bf{h}})$ be the induced module (sometimes called the standard module corresponding to ${\boldsymbol{\la}}$). 
Recall that the algebra $H_{n,r}$ is graded via 
\begin{equation*}
\on{deg}x_j=1,\, \on{deg}y_j=-1,\, \on{deg}\Gamma_n=\on{deg}{\bf{h}}=0.
\end{equation*}
This grading induces a grading on $\Delta(\boldsymbol{\la})$. 
The following lemma describes all the {\em{graded}} endomorphisms of our modules $\Delta(\bla)$.
\begin{Lemma}
We have $\on{End}^{\mathrm{gr}}_{H_{n,r}}(\Delta(\boldsymbol{\la}))={\bf{h}}$. 
\begin{proof}
Note that $\Delta({\boldsymbol{\la}})_0=S({\boldsymbol{\la}}) \otimes {\bf{h}}$ and $\Delta({\boldsymbol{\la}})_0$ generates $\Delta({\boldsymbol{\la}})$ over $H_{n,r}$. Since we consider graded isomorphisms,  $\Delta({\boldsymbol{\la}})_0$ is mapped to itself. Now the claim follows from the equality $\on{End}_{{\bf{h}} \otimes \BC\Gamma_n}(S({\boldsymbol{\la}}) \otimes {\bf{h}})={\bf{h}}$.
\end{proof}

\end{Lemma}

Recall the algebra 
\begin{multline*}
Q_{n,r}:=Z(H_{n,r})_0/\sum_{i>0} Z(H_{n,r})_{-i} Z(H_{n,r})_{i}=
\\ =Z(H_{n,r})/(b \in Z(H_{n,r})_i,\, i \neq 0)
\end{multline*}
of functions on schematic fixed points $(\on{Spec}Z(H_{n,r}))^{\BT}$.

Note that the action of $Z(H_{n,r})_0$ on $\Delta(\boldsymbol{\la})_0$ factors through $Q_{n,r}$ so we obtain a homomorphism 

\begin{equation*}
\phi\colon Q_{n,r} \ra \bigoplus_{\boldsymbol{\la}}\on{End}^{\mathrm{gr}}_{H_{{n,r}}}(\Delta({\boldsymbol{\la}}))=\bigoplus_{\bla}\on{\bf{h}}=E.
\end{equation*}

\begin{Rem}
\emph{In \cite{GO0} Gordon defines so-called baby Verma modules over certain quotients of $H_{n,r}$ (called restricted Cherednik algebras of $\Gamma_n$). Homomorphism $\phi$ can also be  defined  by replacing $\Delta({\boldsymbol{\la}})$ by the (universal analogs) of baby Verma modules and acting by $Q_{n,r}$ on them (compare with \cite{MM}).}
\end{Rem}

We have constructed the homomorphism $\phi$, let us now describe generators of the algebra $Q_{n,r}$.

\begin{Def}\label{OD_elements_DEF!}
Dunkl-Opdam operators $u_i$, $i=1,\ldots,n$ are the following elements of $H_{n,r}$:
\begin{equation}\label{DO_forms}
u_i:=\frac{1}{r}y_ix_i-\kappa \sum_{j>i}(ij)\zeta_{i,j}+p(\epsilon_i)=\frac{1}{r}x_iy_i+\kappa JM_{\Gamma_n,i}+p(\eta^{-1}\epsilon_i),
\end{equation}
where the last equality follows from the fact that $$[x_i,y_i]=-\kappa\sum_{j \neq i}\sum_{p=0}^{r-1}(ij)\epsilon_i^p\epsilon_j^{-p}-c(\epsilon_i)$$  together with the equality $c(\epsilon_i)=rp(\eta^{-1} \epsilon_i)-rp(\epsilon_i)$.

\end{Def}

\begin{Rem}

\emph {Note  that the element $u_i$ identifies with the element $\frac{z_i}{r}$ defined in \cite[Section 3.2]{MM}.}
\end{Rem}

\begin{Lemma}\label{symm_dunkl_central}
The subalgebra ${\bf{h}}[u_1,\ldots,u_n]^{S_n} \subset H_{n,r}$ is central.
\end{Lemma}
\begin{proof}
This follows from~\cite[Theorem 3.4]{MM}, the statement can be also deduced from the presentation of $H_{n,r}$ given in~\cite{webster}.

\end{proof}

The following lemma describes generators of the algebra $Q_{n,r}$, see Proposition \ref{conj_cartan_surj} 
for an alternative argument (that covers a much more general situation). We are grateful to Gwyn Bellamy for explaining this lemma to us.
\begin{Lemma}\label{gen_of_Q!}
Classes of elements $e_k(u_1,\ldots,u_n)$, $k=1,\ldots,n$ generate the algebra $Q_{n,r}$.
\end{Lemma}
\begin{proof}
The claim follows from the proof of ~\cite[Theorem 5.5]{MM}. Let us repeat the argument. Recall that $Q_{n,r}$ is the quotient of the algebra $Z(H_{n,r})$. Consider the following filtration on $H_{n,r}$:
\begin{equation*}
\on{deg}x_i=\on{deg}y_i=1,\, \on{deg}\kappa=\on{deg}c_j=\on{deg}\Gamma_n=0. 
\end{equation*}
We have 
\begin{equation*}
\on{gr}H_{n,r}= \Big(\BC \Gamma_n \ltimes S^{\bullet}(\mathfrak{h} \oplus \mathfrak{h}^*)\Big) \otimes {\bf{h}}.
\end{equation*}

We need to prove that classes of the elements $\sum_{i=1}^nu_i^{ra+c}$, $a,c \in \BZ_{\geqslant 0}$ do generate the algebra $Q_{n,r}$. To see that it is enough to show that the elements $\on{gr}\Big(\sum_{i=1}^{n}u_i^{ra+c}\Big)=\on{gr}\Big(\sum_i(x_iy_i)^{ra+c}\Big)$ do generate $\on{gr}Q_{n,r}$.Note  that (by Theorem 3.3 from \cite{EG}) $\on{gr}Q_{n,r}$ is the quotient of
\begin{multline*}
Z(H_{n,r})={\bf{h}}[x_1,\ldots,x_n,y_1,\ldots,y_n]^{S_n \ltimes (\BZ/r\BZ)^n}=\\
={\bf{h}}[x_1y_1,\ldots,x_ny_n,x_1^r,\ldots,x_n^r,y_1^r,\ldots,y_n^r].
\end{multline*}
In particular, it is generated by $\{\sum_{i=1}^{n}(x_i^r)^a(y_i^r)^b(x_iy_i)^c,\,a,b,c \in \BZ_{\geqslant 0}\}$ \newline (\cite[Lemma~11.17]{EG}, \cite[Lemma~1]{wang} see also Section~\ref{fixed_points_are_free_section} for more detailed discussion of the generators of $\on{gr}Q_{n,r}$).
It remains to note that the class of $\sum_{i=1}^{n}(x_i^r)^a(y_i^r)^b(x_iy_i)^c$ in $\on{gr}Q_{n,r}$ is zero if $a \neq b$ and for $a=b$ we have $\sum_{i=1}^{n}(x_i^r)^a(y_i^r)^a(x_iy_i)^c=\sum_{i=1}^n(x_iy_i)^{ra+c}$.
The claim follows.
\end{proof}

\begin{Lemma}\label{act_u_L_0}
Let $B$ be a Young tableau on $\boldsymbol{\la}$ and recall that $p_B \in S(\boldsymbol{\la})$ is the corresponding vector.
The element $u_i$ acts on $p_{B}$ via the multiplication by 
\begin{equation*}
\kappa \on{ct}(B(i))+p(\eta^{\beta(B(i))-1}).
\end{equation*}
\end{Lemma}
\begin{proof}
Follows from
the definition of $u_i$ (see~\ref{DO_forms}) together with~(\ref{act_JM_S_la}).
\end{proof}

\begin{Prop}\label{prop_inj_Q_use}
The homomorphism $\phi\colon Q_{n,r} \ra \bigoplus_{\boldsymbol{\la}}\on{End}_{H_{n,r}}^{\mathrm{gr}}(\Delta(\boldsymbol{\la}))$ becomes isomorphism after tensoring by $F=\on{Frac}({\bf{h}})$. This homomorphism is injective. It sends generators $[e_k(u_1,\ldots,u_n)]$ to the collection 
\begin{multline*}
(e_{k}(\kappa c^0_1+p(1),\ldots,\kappa c^{0}_{|\la^0|}+ p(1),\ldots \\ \ldots, \kappa c^{r-1}_1+p(\eta^{r-1}),\ldots,\kappa c^{r-1}_{|\la^{r-1}|}+p(\eta^{r-1})))_{\boldsymbol{\la} \in \CP(r,n)},
\end{multline*}
where $c^l_1,c^l_2,\ldots,c^l_{|\la^l|}$ is the multiset of contents of boxes of the diagram $\BY(\la^l)$.
\end{Prop}
\begin{proof}
The proof repeats the argument in the proof of Proposition~\ref{psi_tilda_inj!}. The only difference is that we use Appendix \ref{fixed_points_are_free_section} (flatness of $Q_{n,r}$ over ${\bf{h}}$)  together with   the fact that the fiber of $Q_{n,r}$ over a generic point is $\BC^{\oplus |\CP(r,n)|}$ (follows from \cite{GO0}) instead of \cite[Theorem 1]{Br} and Lemma~\ref{act_u_L_0} instead of Proposition ~\ref{central_act_via_cyclotom_rep}.
\end{proof}

\section{Proof of Theorem~\ref{MAIN_THEOREM}}\label{proof_of_main_section}

\subsection{Proof of Theorem~\ref{MAIN_THEOREM}}
Let us now recall the statement of Theorem~\ref{MAIN_THEOREM}.
\begin{Th}\label{prf_of_main_th_arg}
After the identification of parameters~(\ref{identify_all_parameters}) the graded algebras 
\begin{equation*}
H^*_A(\mathfrak{M}(n,r)),\, Z(R^r(n))^{JM},\, Q_{n,r},\, \BC[\on{Spec}(H^{A \times (\on{GL}_n)_{\CO}}_*(\CR_{n,r}))^{\mathbb{T}}] 
\end{equation*}
are isomorphic. Isomorphisms above identify generators as follows: 
\begin{equation}\label{fun_gen}
c_k(\mathcal{V})=[e_k(z_1,\ldots,z_n)]=[e_k(u_1,\ldots,u_n)]=[m_k].
\end{equation}
\end{Th}
\begin{proof}
The claim follows from Propositions~\ref{emb_cohom_prop_use}, \ref{psi_tilda_inj!}, \ref{prop_inj_Q_use}. The equality $e_k(u_1,\ldots,u_n)=m_k$ follows from \cite[Section 4]{webster}.
\end{proof}

\begin{Rem}
{\emph { Recall that $\mathfrak{M}_0(n,r)^!$ and its deformation have  realization as certain quiver varieties (see Remark \ref{dual_as_quiv_remark}). It is a natural question to describe functions $m_k$ in ``quiver terms''. 
For $r=1$ these functions actually appear in Section \ref{Hikita_Hilbert} and are given by $(X,Y,\gamma,\delta) \mapsto e_k(\al_1,\ldots,\al_n)$, where $\al_1,\ldots,\al_n$ are roots of the characteristic polynomial of $YX \in \on{End}(V)$. Functions $m_k$ can be similarly described for arbitrary $r$, see \cite[Section 6]{TO} for details (for $r=l>1$  there is a minor computational error in this paper that should be fixed, that's why we decided not to go into details here). Let us finally mention that the algebra generated by $\{m_k,~k=1,\ldots,n\}$ defines an integrable system on the Coulomb branch $\CM(n,r)$. This Coulomb branch can be identified with a certain Cherkis bow variety and this integrable system can be naturally described in these terms (see \cite{NAYU} for details, integrable systems are denoted by $\varpi_C$, $\Psi$ in loc. cit.).}}

\end{Rem}

\section{Possible generalizations}\label{general_section!}

\subsection{Changing the quiver} Let $I=(I_0,I_1)$ be a finite quiver and let ${\bf{v}}=(v_i)_{i \in I_0}$ be a  dimension vector for $I$ and  ${\bf{w}}=(w_i)_{i \in I_0}$ be a framing. We can consider the corresponding (smooth) Nakajima quiver variety that we denote  $\mathfrak{M}({\bf{v}},{\bf{w}})$ (see Section \ref{naka_quiver_def}). There is a natural torus $A$ acting  on $\mathfrak{M}({\bf{v}},{\bf{w}})$. We can  consider the symplectically dual variety $\mathcal{M}({\bf{v}},{\bf{w}})_{\mathfrak{a}}$ that can be described as the spectrum of the algebra $H^{A \times (G_{\bf{v}})_{\CO}}_*(\CR_{{\bf{v}},{\bf{w}}})$ of equivariant homology of the variety of triples $\CR_{{\bf{v}},{\bf{w}}}$ corresponding to $I$, ${\bf{v}}$, ${\bf{w}}$ (see \cite{BFN}). Let $\BT$ be $(\BC^\times)^{|I_0|}$. We have a natural action of $\BT$ on $H_*^{A \times (G_{\bf{v}})_{\CO}}(\CR_{{\bf{v}},{\bf{w}}})$ (see \cite[Section 3(v)]{BFN}). 
Recall that $\mathfrak{M}({\bf{v}},{\bf{w}})=\mu^{-1}(0)^{\mathrm{st}}/G_{\bf{v}}$ so  $H^*_A(\mathfrak{M}({\bf{v}},{\bf{w}}),\BC)=H^*_{A \times G_{\bf{v}}}(\mu^{-1}(0)^{\mathrm{st}},\BC)$ and we have the natural {\em{surjective}} (see \cite{MN}) homomorphism: 
\begin{equation}\label{hom_for_quiv!}
H^*_{A \times G_{\bf{v}}}(\on{pt},\BC) \twoheadrightarrow    H^*_{A \times G_{\bf{v}}}(\mu^{-1}(0)^{\mathrm{st}},\BC)=H^*_A(\mathfrak{M}({\bf{v}},{\bf{w}}),\BC).
\end{equation}
Similarly, we can consider the natural composition 
\begin{equation}\label{hom_for-Coulomb!}
H^*_{A \times G_{\bf{v}}}(\on{pt},\BC) \subset   H_*^{A \times (G_{\bf{v}})_{\CO}}(\CR_{{\bf{v}},{\bf{w}}}) \twoheadrightarrow \BC[ \on{Spec}(H_*^{A \times (G_{\bf{v}})_{\CO}}(\CR_{{\bf{v}},{\bf{w}}}))^{\BT}]
\end{equation}
that is also surjective (see Proposition \ref{conj_cartan_surj} below). 

The equivariant version of the following conjecture was formulated in \cite[Conjecture 5.26]{HKW}.

\begin{Conj}\label{Conj_HN_expl}
The kernel of (\ref{hom_for_quiv!}) is equal to the kernel of (\ref{hom_for-Coulomb!}) i.e.  they define the same closed subschemes of $\on{Spec}H^*_{A \times G_{\bf{v}}}(\on{pt},\BC)$.
\end{Conj}

\begin{Rem}
{\emph{If $I$ is a Jordan quiver then Conjecture~\ref{Conj_HN_expl} follows from Theorem~\ref{MAIN_THEOREM}. 
}}
\end{Rem}

The approach used in this paper has a chance to be generalized to (some) other quivers. We start with the following conjecture.

\begin{Conj}\label{conj_flatness_gen}
The algebra $\BC[\on{Spec}( H_*^{A \times (G_{\bf{v}})_{\CO}}(\CR_{{\bf{v}},{\bf{w}}}))^{\BT}]$ is flat over $\BC[\mathfrak{a}]$.
\end{Conj}

\begin{Rem}
{\emph{Note that when $I$ is the Jordan quiver then Conjecture \ref{conj_flatness_gen} follows from Appendix \ref{fixed_points_are_free_section}, note also that the ``dual'' statement to the Conjecture \ref{conj_flatness_gen} is indeed true and holds by \cite[Theorem 7.3.5]{NA_quant_affine}.}}
\end{Rem}

\begin{Rem}
{\emph For our purposes, it is enough to prove that \newline $\BC[ \on{Spec}( H_*^{A \times (G_{\bf{v}})_{\CO}}(\CR_{{\bf{v}},{\bf{w}}}))^{\BT}]$ is torsion free over $\BC[\mathfrak{a}]$.}
\end{Rem}

\begin{Warning}
The statement analogous to the Conjecture \ref{conj_flatness_gen} may fail for arbitrary symplectic singularities.   
\end{Warning}

Recall now that by \cite{MN} the algebra $H^*_A(\mathfrak{M}({\bf{v}},{\bf{w}}))$ that appears at the LHS of Conjecture \ref{Conj_HN_expl} is generated over $H^*_A(\on{pt})$ by the Chern classes of tautological bundles $c_k(\CV_i)$. It turns out that the  ``dual'' statement can also be proven (without any restrictions on the quiver $I$) i.e., that the algebra $\BC[\on{Spec}(H_*^{A \times (G_{\bf{v}})_{\CO}}(\CR_I({\bf{v}},{\bf{w}})))^{\BT}]$ of schematic fixed points is generated over $H^*_A(\on{pt})$ by the classes of $m_{k,i}:=c_k(V_i) * 1$. This is equivalent to the following proposition, whose proof was explained to us by Ben Webster and Alex Weekes and  will appear in their joint work with Joel Kamnitzer and Oded Yacobi.

\begin{Prop}\label{conj_cartan_surj}
The natural embedding $H^*_{A \times G_{\bf{v}}}(\on{pt}) \subset H_*^{A \times (G_{\bf{v}})_{\CO}}(\CR_{{\bf{v}},{\bf{w}}})$ induces surjection $H^*_{A \times G_{\bf{v}}}(\on{pt}) \twoheadrightarrow \BC[(\on{Spec}H_*^{A \times (G_{\bf{v}})_{\CO}}(\CR_{{\bf{v}},{\bf{w}}}))^{\BT}]$.
\end{Prop}
\begin{proof}
The claim follows from \cite[Proposition 3.1]{weekes1} (see also  \cite[Remark 6.7]{BFN}) using that the dressed minuscule monopole operators have a nonzero degree with respect to $\BT$ so their images  in $\BC[(\on{Spec}H_*^{A \times (G_{\bf{v}})_{\CO}}(\CR_{{\bf{v}},{\bf{w}}}))^{\BT}]$ are zero.  
\end{proof}

Assume now that Conjecture~\ref{conj_flatness_gen} holds. Assume also that there exists a resolution of singularities $\widetilde{\mathcal{M}}({\bf{v}},{\bf{w}})_{\mathfrak{a}} \ra \mathcal{M}({\bf{v}},{\bf{w}})_{\mathfrak{a}}$  (see \cite{weekes} for the discussion). It induces the morphism $\widetilde{\mathcal{M}}({\bf{v}},{\bf{w}})_{\mathfrak{a}}^{\BT} \ra \mathcal{M}({\bf{v}},{\bf{w}})_{\mathfrak{a}}^{\BT}$ that gives us the embedding $\BC[\mathcal{M}({\bf{v}},{\bf{w}})_{\mathfrak{a}}^{\BT}] \subset \BC [ \widetilde{\mathcal{M}}({\bf{v}},{\bf{w}})_{\mathfrak{a}}^{\BT}]$, the fact that this is indeed an embedding follows from Conjecture \ref{conj_flatness_gen}. 
We can also consider the embedding $H^*_A(\mathfrak{M}({\bf{v}},{\bf{w}})) \subset H^*_A(\mathfrak{M}({\bf{v}},{\bf{w}})^A)$ that corresponds to the restriction $\mathfrak{M}({\bf{v}},{\bf{w}})^A \subset \mathfrak{M}({\bf{v}},{\bf{w}})$. Using that $c_k(\CV_i)$, $m_{k,i}$ are generators of our algebras it then remains to show that the images of $c_k(\CV_i)$, $m_{k,i}$ under these embeddings  coincide.

\begin{Rem}
\emph{Let us point out that even without assuming Conjecture \ref{conj_flatness_gen} one can consider the image of $\BC[\mathcal{M}({\bf{v}},{\bf{w}})_{\mathfrak{a}}^{\BT}] \ra \BC[\widetilde{\mathcal{M}}({\bf{v}},{\bf{w}})_{\mathfrak{a}}^{\BT}]$ and if the images of $c_k(\CV_i)$, $m_{k,i}$ coincide then we obtain a surjective homomorphism of algebras  \newline $\BC[\mathcal{M}({\bf{v}},{\bf{w}})_{\mathfrak{a}}^{\BT}] \twoheadrightarrow H^*_A(\mathfrak{M}({\bf{v}},{\bf{w}}),\BC)$ that is an isomorphism generically. Conjecture \ref{conj_flatness_gen} implies that this surjection is an isomorphism.}
\end{Rem}

Let us finally  note that to show that the images of $c_k(\CV_i)$, $m_{k,i}$ coincide, it is enough to do the following. For a generic ${\bf{a}} \in \mathfrak{a}$ we need to construct a bijection 
\begin{equation*}
\mathfrak{M}({\bf{v}},{\bf{w}})^A \iso \CM({\bf{v}},{\bf{w}})_{\bf{a}}^{\BT},\, p \mapsto p'  
\end{equation*}
such that 
\begin{equation*}
e_k(\al_1,\ldots,\al_{v_i})=m_{k,i}(p'),
\end{equation*}
where $\al_1,\ldots,\al_{v_i}$ are eigenvalues of ${\bf{a}} \in \mathfrak{a}$ acting on $\CV_i|_{p'}$ and $m_{k,i}$ is considered as a function on $\CM({\bf{v}},{\bf{w}})$. In this paper we do exactly this for $\mathfrak{M}({\bf{v}},{\bf{w}})=\mathfrak{M}(n,r)$.

\appendix

\section{Flatness of schematic fixed points: approach of Hikita and Hatano}\label{fixed_points_are_free_section}
The goal of this appendix is to give a self-contained proof of the fact that the algebra $Q_{n,r}$ of functions on the schematic fixed points 
\begin{equation*}
(\on{Spec}Z(H_{n,r}))^{\BT} = \mathcal{M}(n,r)_{\mathfrak{a}}^{\BT}
\end{equation*}
is a flat (hence, free) ${\bf{h}}$-module of rank $|\CP(r,n)|$. Let us first of all note that by the graded Nakayama lemma together with the fact that $\on{dim}_{F}(Q_{n,r} \otimes_{{\bf{h}}} F)=|\CP(r,n)|$ (this follows from \cite{GO0}, see also \cite[Section 5]{TO}) in order to prove this fact it is enough to show that 
\begin{equation*}
\on{dim}_{\BC} Q_{n,r}/(\kappa,c_1,\ldots,c_{r-1}) \leqslant |\CP(r,n)|
\end{equation*} 
i.e. that 
\begin{equation}\label{dim_fiber_main_estimate}
\on{dim}_{\BC} \BC[(\BA^{2n}/\Gamma_n)^{\BT}] \leqslant |\CP(r,n)|.    
\end{equation}
The goal of this section is to prove the inequality (\ref{dim_fiber_main_estimate}). Our argument simply follows papers ~\cite{HI} (for $r=1$ case) and ~\cite{HA} (in general) but is much shorter since we do not need any explicit formulas for the multiplication rule of the elements of $\BC[(\BA^{2n}/\Gamma_n)^{\BT}]$ and only need to estimate the dimension of this algebra from above since the estimate  from below follows from the deformation argument (so we only need~\cite[Lemma~2.5]{HI} for $r=1$ case and ~\cite[Lemma 2.1.4]{HA} for general $r$).
We start from the case $r=1$ i.e. from the case when $\Gamma_n=S_n$.
\subsection{Hilbert scheme case ($r=1$)}\label{hilb_scheme_flatness}

Let us recall some notation (we follow \cite{HI}).

\begin{Def}
An unordered sequence $\La=(a_1,b_1)\ldots (a_l,b_l)$ with $(a_i,b_i) \in \BZ_{\geqslant 0}^{2} \setminus \{(0,0)\}$ is called bipartite partition of $(a,b) \in \BZ^2_{\geqslant 0} \setminus \{(0,0)\}$ if $\sum_{i=1}^l a_i=a$, $\sum_{i=1}^l b_i=b$. We set $\ell(\La)=l$, $|\La|=(a,b)$.
\end{Def}

We have a natural surjection 
\begin{multline*}
\BC[S^{n+1}(\BA^2)]=\BC[x_1,\ldots,x_{n+1},y_1,\ldots,y_{n+1}]^{S_{n+1}} \twoheadrightarrow \\ \twoheadrightarrow \BC[x_1,\ldots,x_n,y_1,\ldots,y_n]^{S_n}=\BC[S^n(\BA^2)]    
\end{multline*}
and denote by $S$ the inverse limit (in the category of graded algebras)
\begin{equation*}
S:=\underset{\longleftarrow}{\on{lim}}\, \BC[S^n(\BA^2)].  
\end{equation*}

We denote by $m_{\La} \in S$ the symmetrization of the monomial $x_1^{a_1}y_1^{b_1}\ldots x_l^{a_l}y_l^{b_l}$. 
We have:
\begin{equation*}
S=\BC[m_{(a,b)}\,|\, (a,b) \in \BZ_{\geqslant 0}^2 \setminus \{(0,0)\}].    
\end{equation*}
For $(a,b) \in \BZ_{\geqslant 0}^2 \setminus \{(0,0)\}$ we set $(a,b)\La:=(a,b)(a_1,b_1)\ldots (a_l,b_l)$. If $(a,b)=(a_i,b_i)$ for some $i \in \{1,2,\ldots,l\}$ we set 
\begin{equation*}
\La \setminus (a,b):=(a_1,b_1)\ldots (a_{i-1},b_{i-1})(a_{i+1},b_{i+1})\ldots (a_l,b_l).
\end{equation*}
We set $\ol{S}:=S/(m_{(a,b)},\, a \neq b)$ and denote by $\ol{m}_{\La}$ the image of $m_{\La}$ in $\ol{S}$. Note that directly from the definitions for every $n \in \BZ_{\geqslant 1}$ we have a surjective homomorphism $\ol{S} \twoheadrightarrow \BC[(S^n(\BA^2))^{\BT}]$ which sends every $\ol{m}_{\La}$ with $\ell(\La)>n$ to zero.
The following lemma is clear.
\begin{Lemma}\label{mult_by_a_b}
Let $\La$ be a bipartite partition. 
We have 

\begin{equation*}
m_{(a,b)}m_{\La}=km_{(a,b)\La}+\sum_{(i,j) \in \La}k_{(i,j)}m_{(a+i,b+j)\La \setminus (i,j)}
\end{equation*}
for some $k, k_{(i,j)} \in \BZ_{>0}$.

\end{Lemma}

For a partition $\la=(\la_1,\la_2,\ldots,\la_l)$ we denote by $(\la,0)$ the bipartite partition $(\la_1,0),\ldots,(\la_l,0)$. 

This lemma is~\cite[Lemma 2.5]{HI}.
\begin{Lemma}\label{basis_of_ol_S}
$\{\ol{m}_{(\la,0)(0,1)^{|\la|}}\,|\, \la~\text{is partition}\}$ spans $\ol{S}$.
\end{Lemma}
\begin{proof}
Let us first of all note that the functions $\ol{m}_{(a,a)(b,b)(c,c)\ldots}$ span $\ol{S}$. Indeed to prove this it is enough to show that every $\ol{m}_{\La}$ can be obtained as a linear combination of  $\ol{m}_{(a,a)(b,b)(c,c)\ldots}$. This can be proved by the induction on $\ell(\La)$ using Lemma \ref{mult_by_a_b} together with the fact that $\ol{m}_{(a,b)}=0$ for $a \neq b$.

It remains to show that every $\ol{m}_{(a,a)(b,b)(c,c)\ldots}$ can be expanded in terms of $\ol{m}_{(\la,0)(0,1)^{|\la|}}$.
To see that it is enough to show that every $\ol{m}_{(a_1,b_1)\ldots (a_l,b_l)(0,1)^m}$ with $a_i \geqslant b_i$ can be obtained as a linear combination of $\ol{m}_{(\la,0)(0,1)^{|\la|}}$. We prove this by the induction on $d=\sum_{i=1}^l b_i$. For $d=0$ the claim is clear. 

For the induction step without losing the generality, we can assume that $b_1>0$. Using Lemma~\ref{mult_by_a_b} we obtain:
\begin{multline*}
m_{(a_1,b_1-1)}m_{(a_2,b_2)\ldots (a_l,b_l)(0,1)^{m+1}}
=k_1m_{(a_1,b_1)\ldots (a_l,b_l)(0,1)^{m}}+\\+k_0m_{(a_1,b_1-1)(a_2,b_2)\ldots (a_l,b_l)(0,1)^{m+1}}+
\sum_{i=2}^l k_i m_{(a_2,b_2)\ldots (a_1+b_i,b_1+b_i-1)\ldots(a_l,b_l)(0,1)^{m+1}}  
\end{multline*}
for some $k_0,k_1,\ldots,k_l \in \BZ$ with $k_1 \neq 0$. Induction hypothesis together with the fact that $\ol{m}_{(a_1,b_1-1)}$ finish the proof.
\end{proof}

\begin{Cor}
The image of $\{\ol{m}_{(\la,0)(0,1)^{|\la|}},|\, \ell(\la)+|\la| \leqslant n\}$ spans $\BC[(S^n(\BA^2))^{\BT}]=\BC[(\BA^{2n}/S_n)^{\BT}]$. In particular, we have 
\begin{equation*}
\on{dim}\BC[(\BA^{2n}/S_n)^{\BT}] \leqslant |\CP(n)|.    
\end{equation*}
\end{Cor}
\begin{proof}
Clearly the elements $\ol{m}_{(\la,0)(0,1)^{|\la|}}$ with  $\ell(\la)+|\la| > n$ lie in the kernel of $\ol{S} \twoheadrightarrow \BC[(S^n(\BA^2))^{\BT}]$. Now the first claim follows from Lemma \ref{basis_of_ol_S}. It remains to note that we have a bijection 
\begin{equation*}
\{\la\,|\, \ell(\la)+|\la| \leqslant n\} \iso \CP(n)    
\end{equation*}
that sends a partition $\la=(1^{\al_1}2^{\al_2}\ldots )$ to the partition $\hat{\la} \in \CP(n)$ given by 
\begin{equation*}
\hat{\la}=1^{n-\ell(\la)-|\la|}2^{\al_1}3^{\al_2}\ldots i^{\al_{i-1}} \ldots.
\end{equation*}
\end{proof}

\subsection{General case ($r$ is arbitrary)}

Let us now generalize the arguments of Section~\ref{hilb_scheme_flatness} to the arbitrary $r \in \BZ_{\geqslant 1}$. We follow~\cite{HA}.
We start with some notation. Recall that $\BC[\BA^{2n}/\Gamma_n]$ is nothing else but 
\begin{multline*}
\BC[x_1,\ldots,x_n,y_1,\ldots,y_n]^{S_n \ltimes (\BZ/r\BZ)^n} \simeq \\
\Big(\BC[x'_1,\ldots,x'_n,y'_1,\ldots,y'_n,z'_1,\ldots,z'_n]/(x'_1y'_1-(z'_1)^r,\ldots,x'_ny'_n-(z'_n)^r)\Big)^{S_n}    
\end{multline*}
where the isomorphism is given by 
\begin{equation*}
x_i' \mapsto x_i^r,\, 
y_i' \mapsto y_i^r,\,
z'_i \mapsto x_iy_i.
\end{equation*}

\begin{Rem}
{\em{Geometrically isomorphism above corresponds to the identification $\BC[\BA^{2n}/\Gamma_n] \simeq S^n(\BA^2/(\BZ/r\BZ))$.}}
\end{Rem}
Set  \begin{multline*}
    I_n :=(x'_1y'_1-(z'_1)^r,\ldots,x'_ny'_n-(z'_n)^r)\subset \\ \subset \BC[x'_1,\ldots,x'_n,y'_1,\ldots,y'_n,z'_1,\ldots,z'_n], \end{multline*}
and
\begin{equation*}
S':=\underset{\longleftarrow}{\on{lim}}\, \BC[x'_1,\ldots,x'_n,y'_1,\ldots,y'_n,z'_1,\ldots,z'_n]^{S_n},\, I:=\underset{\longleftarrow}{\on{lim}}\,I_n,\, S:=S'/I.
\end{equation*}

For every tripartition $\La=(a_1,b_1,c_1)(a_2,b_2,c_2),\ldots,(a_l,b_l,c_l)$ let $m'_\la \in S'$ be the symmetrization of the monomial \newline $(x_1')^{a_1}(y_1')^{b_1}(z_1')^{c_1}\ldots (x_l')^{a_l}(y_l')^{b_l}(z_l')^{c_l}$, we set $\ell(\La):=l$.
We denote by $m_\La \in S$ the image of $m'_\La$.
The following lemma is clear.
\begin{Lemma}
The set $\{m'_{\La}\,|\, \La \text{-tripartition}\}$ spans $S'$. So 
\begin{equation*}
\{m_{\La}\,|\,\La=(a_1,b_1,c_1)\ldots,\, c_i \leqslant r-1\}
\end{equation*}
spans $S$.
\end{Lemma}

\begin{Lemma}\label{mult_by_abc}
Let $\La$ be a tripartition and $(a,b,c) \in \BZ^3_{\geqslant 0} \setminus \{(0,0,0)\}$. Then we have 
\begin{equation*}
m'_{(a,b,c)}m'_{\La}=? \cdot m'_{(a,b,c)\La} + \sum_{(i,j,k) \in \La} ? \cdot m'_{(a+i,b+j,c+k)\La \setminus (i,j,k)}.  
\end{equation*}
with $?$ being some positive numbers.
As a corollary  we have
\begin{multline*}
m_{(a,b,c)}m_{\La} =\,? \cdot m_{(a,b,c)\La} + \sum_{(i,j,k) \in \La,\, c+k \leqslant r-1} ? \cdot m_{(a+i,b+j,c+k)\La \setminus (i,j,k)}+\\
+\sum_{(i,j,k) \in \La,\, c+k \geqslant r} ? \cdot m_{(a+i+1,b+j+1,c+k-r)\La \setminus (i,j,k)}.
\end{multline*}
with $?$ being some positive numbers.  
\end{Lemma}

For $\La=(a_1,b_1,c_1)\ldots (a_l,b_l,c_l)$ we set $\on{deg}\La:=\sum_{i=1}^la_i-\sum_{i=1}^{l}b_i$.
Let $J \subset S$ be the ideal generated by $\{m_{\La}\,|\, \on{deg}\La \neq 0\}$. We set $\ol{S}:=S/J$.
We denote by $\ol{m}_{\La} \in \ol{S}$ the image of $m_{\La}$.

For an $r$-tuple of partitions ${\boldsymbol{\la}}=(\la^0,\la^1,\ldots,\la^{r-1})$ and a collection of nonnegative numbers ${\bf{p}}=(p_1,\ldots,p_{r-1})$ we define tripartition to be denoted by the symbol
\begin{multline*}
(\boldsymbol{\la},{\bf{p}}):=(\la^0_1,0,0)\ldots (\la^0_{\ell(\la^0)},0,0),(0,0,1)^{p_1},(\la_1^1,0,1),\ldots,(\la^1_{\ell(\la^1)},0,1),\ldots\\ \ldots,(0,0,r-1)^{p_{r-1}},(\la^{r-1}_1,0,r-1),\ldots,(\la^{r-1}_{\ell(\la^{r-1})},0,r-1).
\end{multline*}
Recall that $\ell({\boldsymbol{\la}})=\sum_{i=0}^{r-1}\ell(\la^i)$, $|{\boldsymbol{\la}}|=\sum_{i=0}^{r-1}|\la^i|$. We set $|{\bf{p}}|:=p_1+\ldots+p_{r-1}$.
The following lemma is \cite[Lemma 2.1.4]{HA}.
\footnote{Hatano's statement had a minor typo, Lemma \ref{span_by_bold_lambda} is a corrected version of \cite[Lemma 2.1.4]{HA}.}
\begin{Lemma}\label{span_by_bold_lambda}
The set $\{\ol{m}_{({\boldsymbol{\la}},{\bf{p}})(0,1,0)^{|{\boldsymbol{\la}}|}}\}$ spans $\ol{S}$.
\end{Lemma}
\begin{proof}
We start from proving that elements $\ol{m}_{(a_1,a_1,c_1)(a_2,a_2,c_2)\ldots }$ span $\ol{S}$.  It is enough to show that every $m_{\La}$ can be presented as a linear combination of  \newline $\ol{m}_{(a_1,a_1,c_1)(a_2,a_2,c_2)\ldots }$. We can assume that there exists $(a,b,c) \in \La$ such that $a \neq b$ (otherwise there is nothing to prove).  From Lemma~\ref{mult_by_abc} it follows that 
\begin{equation*}
0=\ol{m}_{(a,b,c)}\ol{m}_{\La \setminus (a,b,c)}=k\ol{m}_{\La}+\sum_{\La',\, \ell(\La')=\ell(\La)-1} ? \cdot    \ol{m}_{\La'} 
\end{equation*}
and the claim follows by the induction on the length of $\La$. 

It remains to show that every element $\ol{m}_{(a_1,a_1,c_1)\ldots (a_l,a_l,c_l)}$ can be written as a linear combination of $\ol{m}_{({\boldsymbol{\la}},{\bf{p}})(0,1,0)^{|{\boldsymbol{\la}}|}}$. We prove a more general statement: that every element  $\ol{m}_{(a_1,b_1,c_1)\ldots (a_l,b_l,c_l)(0,1,0)^{k}}$ with $a_i \geqslant b_i$ can be written as a linear combination of $\ol{m}_{({\boldsymbol{\la}},{\bf{p}})(0,1,0)^{|{\boldsymbol{\la}}|}}$. We prove this claim by the induction on $b+l$, where $b:=\sum_{i=1}^l b_i$. Let us, first of all, note that we can assume that $\sum_{i=1}^l a_i=b+k$. For $b=0$ we must have $b_i=0$ for every $i$ and then there is nothing to prove. Suppose now that $b>0$. Without losing the generality we can assume that $b_1>0$. By Lemma~\ref{mult_by_abc} (also using that $a_1 \geqslant b_1 > b_1-1$) we have 
\begin{multline*}
0=\ol{m}_{(a_1,b_1-1,c_1)}\ol{m}_{(a_2,b_2,c_2)\ldots (a_l,b_l,c_l)(0,1,0)^{k+1}}=\\
= ? \cdot \ol{m}_{(a_1,b_1-1,c_1)(a_2,b_2,c_2)\ldots (a_l,b_l,c_l)(0,1,0)^{k+1}}+u\ol{m}_{(a_1,b_1,c_1)\ldots (a_l,b_l,c_l)(0,1,0)^{k}}+\\
+\sum_{c_1+c_i \leqslant r-1}?  \cdot \ol{m}_{(a_2,b_2,c_2)\ldots (a_i+a_1,b_i+b_1-1,c_i+c_1)\ldots (a_l,b_l,c_l)(0,1,0)^{k+1}}+\\
\sum_{c_i+c_1 \geqslant r}? \cdot \ol{m}_{(a_2,b_2,c_2)\ldots (a_i+a_1+1,b_i+b_1,c_i+c_1-r)\ldots (a_l,b_l,c_l)(0,1,0)^{k+1}}
\end{multline*}
with $u \in \BZ_{>0}$.
\end{proof}
Now the claim follows from the induction hypothesis.

\begin{Cor}
The image of the set $\{\ol{m}_{({\boldsymbol{\la}},{\bf{p}})(0,1,0)^{|{\boldsymbol{\la}}|}}\,|\, \ell({\boldsymbol{\la}})+|{\boldsymbol{\la}}|+|{\bf{p}}|\leqslant n\}$ spans $\BC[\BA^{2n}/\Gamma_n]$. In particular 
\begin{equation*}
\on{dim}\BC[\BA^{2n}/\Gamma_n] \leqslant |\CP(r,n)|.    
\end{equation*}
\end{Cor}
\begin{proof}
Note that $\ell(({\boldsymbol{\la}},{\bf{p}})(0,1,0)^{|\boldsymbol{\la}|})=\ell({\boldsymbol{\la}})+|\boldsymbol{\la}|+|{\bf{p}}|$ so the elements $\ol{m}_{({\boldsymbol{\la}},{\bf{p}})(0,1,0)^{|\boldsymbol{\la}|}}$ with $\ell({\boldsymbol{\la}})+|\boldsymbol{\la}|+|{\bf{p}}|>n$ lie in the kernel of $\ol{S} \twoheadrightarrow \BC[\BA^{2n}/\Gamma_n]$. Now  the first claim  follows from Lemma~\ref{span_by_bold_lambda}. It remains to note that we have a bijection 
\begin{equation*}
\{({\boldsymbol{\la}},{\bf{p}})\,|\, \ell({\boldsymbol{\la}})+|\boldsymbol{\la}|+|{\bf{p}}| \leqslant n\} \iso \CP(r,n) \end{equation*}
that sends $({\bla},{\bf{p}})$ with $\la^i=1^{\al_1^i}2^{\al_2^i}\ldots$ to the multipartition $\hat{\bla}$ given by:
\begin{equation*}
\hat{\la}^0=1^{n-\ell(\bla)-|\bla|-|{\bf{p}}|}2^{\al_1^0}3^{\al_2^0}\ldots,~\text{and}~\hat{\la}^i=1^{p_i}2^{\al_1^i}3^{\al_2^i}\ldots~\text{for}~i=1,\ldots,r-1.
\end{equation*}
The inverse map sends an $r$-partition ${\boldsymbol{\mu}}\in \mathcal{P}(r,n)$ with $\mu^i=1^{\beta_1^i}2^{\beta_2^i}\ldots$ to the pair $({\boldsymbol{\la}},{\bf{p}})$ given by:
\begin{equation*}
\la^i=1^{\beta^i_2}2^{\beta^i_3}\ldots k^{\beta^i_{k+1}}\ldots,\, p_j=\beta^j_1~\text{for}~i=0,1,\ldots,r-1,\,j=1,\ldots,r-1.
\end{equation*}

\end{proof}

\end{document}